\documentclass{amsart}
\usepackage{amssymb,mathtools}
\usepackage[hidelinks]{hyperref}
\usepackage{xcolor}
\usepackage{tikz}

\numberwithin{equation}{section}

\newtheorem{theorem}{Theorem}[section]
\newtheorem{proposition}[theorem]{Proposition}
\newtheorem{lemma}[theorem]{Lemma}

\theoremstyle{definition}

\newtheorem{remark}[theorem]{Remark}

\DeclareMathOperator{\Alt}{Alt}

\DeclareMathOperator{\rank}{rank}

\newcommand{\R}{\mathbb R}
\newcommand{\Z}{\mathbb Z}
\newcommand{\ip}[2]{\left\langle #1,#2\right\rangle}
\newcommand{\abs}[1]{\left|#1\right|}
\newcommand{\norm}[1]{\left\lVert#1\right\rVert}
\newcommand{\eps}{\varepsilon}
\newcommand{\dd}{\,\mathrm d}
\newcommand{\cA}{\mathcal A}
\newcommand{\cD}{\mathcal D}

\title[Rootwise estimates for Weyl alternants]
{Rootwise estimates for Weyl alternants}

\author[X. Li]{Xiaocheng Li}
\address{School of Mathematics and Statistics, Nanjing University of Science and Technology, Nanjing, 210094, China}
\email{lixiaocheng@njust.edu.cn}

\author[Y. Zhang]{Yunfeng Zhang}
\address{Department of Mathematics, University of Mississippi,
University, MS 38677}
\email{yzhang19@olemiss.edu}

\dedicatory{On the centenary of Hermann Weyl's character formula}

\date{}
\begin{document}

\begin{abstract}
For a reduced crystallographic root system with Weyl group $W$, Weyl alternants are signed sums of exponentials indexed by $W$. They occur in the numerators of the Weyl character formula and the explicit formula for spherical functions on complex semisimple Lie groups. We establish rootwise bounds that quantify cancellation in these oscillatory sums, with each positive root contributing a factor depending on both the spatial variable and the spectral parameter. More generally, we establish rootwise bounds for all mixed
directional derivatives of the quotient of a Weyl alternant
by the product of positive roots. Our proof extends to arbitrary rank the descent strategy used by the second author for pointwise character bounds on $\mathrm{SU}(3)$. The key ingredients include the decomposition of the Weyl group into parabolic double cosets and the relative BGG--Demazure identity. For spherical functions on complex semisimple groups, these estimates recover the recent pointwise bound of Brumley, Marshall, Matz, and Peterson, and they refine the derivative bounds of Cowling and Nevo by controlling arbitrary mixed radial derivatives uniformly across spatial and spectral root hyperplanes. Moreover, by a localization argument relative to the affine Weyl arrangement, we establish 
rootwise bounds for all mixed radial derivatives of irreducible characters of compact connected semisimple Lie groups.
\end{abstract}

\maketitle

\tableofcontents

\section{Introduction}

\subsection{Weyl alternants}
\label{subsec:intro-alternants}

Let $V$ be a Euclidean space with inner product $\langle\, ,\,\rangle$, and let $V^*$ be the dual space of $V$. Let $\Phi\subset V^*$ be a reduced crystallographic root system.
We do not require $\Phi$ to span $V^*$. 
Let $\Phi^+$ be a positive system, and let
$\Delta=\{\alpha_1,\ldots,\alpha_r\}$ be the corresponding set of simple roots.
We use the inner product  $\langle\, ,\,\rangle$ to identify \(V\) with \(V^*\) when
convenient, and write \(\|\cdot\|\) for the resulting norm on either
space. The
reflection on $V$ associated with $\alpha\in V^*\setminus\{0\}$ is
\[
 s_\alpha X=X-\alpha(X)\alpha^\vee,
 \qquad
 \alpha^\vee=\frac{2\alpha}{\ip{\alpha}{\alpha}}\in V.
\]
Let $W$ denote the Weyl group of $\Phi$.

For the spectral variable $\lambda\in V^*$ and spatial variable $X\in V$, we define the \emph{Weyl alternant}
\begin{equation}\label{eq:linear-alternant}
 A_\lambda^\Phi(X)
  =\sum_{w\in W}\det(w)e^{i \langle w\lambda,X\rangle}.
\end{equation}
For example, when $\Phi$ is of type $A_n$, the Weyl alternant
reduces to the determinant
\[
A_\lambda^\Phi(X)
=\det(e^{i\lambda_j X_k})_{1\leq j,k\leq n+1}.
\]

A fundamental feature of Weyl alternants is their vanishing
on root hyperplanes in both the spatial and spectral variables. 
For type $A_n$, this is evident from the determinant:
if two coordinates of $X$ or $\lambda$ coincide,
two columns or rows of the matrix coincide.
More generally, the vanishing follows from the skew symmetries
under the Weyl group (see \eqref{eq: skew}).
The natural problem is to quantify this vanishing simultaneously
for all roots, uniformly near intersections of root hyperplanes. 
The rank-one case reduces to the elementary inequality
\[
 |\sin (\omega t)|\leq\min\{1,|\omega t|\}, \qquad\omega,t\in\R. 
\]

Our first main result is a pointwise estimate for Weyl alternants
in arbitrary rank. 
We call this estimate \emph{rootwise} because its right-hand side
is the Weyl symmetrization of a product over positive roots, each
factor having the form of the rank-one bound above.

\begin{theorem}[Weyl alternant estimate]
\label{thm:linear}
There is a constant $C_\Phi$ such that, for every
$\lambda\in V^*$ and $X\in V$,
\begin{equation}\label{eq:linear-main}
 \abs{A_\lambda^\Phi(X)}
 \le C_\Phi\sum_{w\in W}
 \prod_{\alpha\in\Phi^+}
 \min\!\left\{1,\,
   \abs{\alpha(X)}
   \abs{\ip{w\lambda}{\alpha^\vee}}
 \right\}.
\end{equation}
\end{theorem}

In fact, to establish the pointwise estimate above, we prove rootwise bounds
for all mixed directional derivatives of the quotient
$A_\lambda^\Phi/\pi$, where
\[
  \pi(X)=\prod_{\alpha\in\Phi^+}\alpha(X).
\]
This quotient, initially defined where $\pi\neq0$, has a canonical
smooth extension to all of $V$; see Lemma~\ref{lem: BGG}.
More generally, our proof involves quotients by root products
associated with parabolic subsystems, as in
Proposition~\ref{prop:local-jets}.
We refer to these quotients collectively as
\emph{Weyl alternant quotients}.

\begin{theorem}[Derivative estimate for Weyl alternant quotients]
\label{thm:jets}
For every integer $k\geq0$, there exists a constant
$C_{\Phi,k}>0$ such that, for all $\lambda\in V^*$,
$X\in V$, and unit vectors $\xi_1,\ldots,\xi_k\in V$,
\begin{equation}\label{eq:global-alternant-jets}
\begin{aligned}
\left|
 D_{\xi_1}\cdots D_{\xi_k}
 \left(\frac{A_\lambda^\Phi}{\pi}\right)(X)
 \right|
\leq
 C_{\Phi,k}\|\lambda\|^k
 \sum_{w\in W}
 \prod_{\alpha\in\Phi^+}
 \frac{
  |\langle w\lambda,\alpha^\vee\rangle|
 }{
  1+|\alpha(X)|\,|\langle w\lambda,\alpha^\vee\rangle|
 }.
\end{aligned}
\end{equation}
\end{theorem}

Taking $k=0$, multiplying by $|\pi(X)|$, and using
$t/(1+t)\asymp\min\{1,t\}$ for $t\geq0$ gives
Theorem~\ref{thm:linear}.

As an illustration, the rank-one alternant quotient is a constant
multiple of $\sin(\omega t)/t$, extended smoothly at $t=0$. 
Theorem~\ref{thm:jets} gives, for every integer $k\geq0$,
\[
 \left|
 \frac{\mathrm d^k}{\mathrm dt^k}
 \left(\frac{\sin(\omega t)}{t}\right)
 \right|
 \leq C_k\frac{|\omega|^{k+1}}{1+|\omega t|},
 \qquad \omega,t\in\mathbb R.
\]
This follows directly from
$\sin(\omega t)/t=\omega\int_0^1\cos(s\omega t)\,\mathrm ds$
when $|\omega t|\leq1$, and from Leibniz's rule when
$|\omega t|\geq1$.

Our main motivation for Theorems~\ref{thm:linear} and~\ref{thm:jets} is to obtain bounds for spherical functions on complex semisimple Lie groups and irreducible characters of compact connected semisimple Lie groups. We discuss these applications next.

\subsection{Spherical functions on complex semisimple Lie groups}
\label{subsec:intro-spherical}

Let $G$ be a complex semisimple Lie group, and let $K$ be a maximal
compact subgroup. Let $\mathfrak g$ and $\mathfrak k$ denote their
Lie algebras respectively, write $\mathfrak g=\mathfrak k\oplus\mathfrak p$ for
the associated Cartan decomposition, and let
$\mathfrak a$ be a maximal abelian subspace of $\mathfrak p$. Let $\Phi$ be the restricted
root system of $G/K$. For $\zeta\in\mathfrak a^*_{\mathbb C}$, write $\zeta=\lambda+i\eta$, with
$\lambda,\eta\in\mathfrak a^*$. We set
$\|\zeta\|=(\|\lambda\|^2+\|\eta\|^2)^{1/2}$.
The inner product $\langle\, ,\, \rangle$ is extended complex bilinearly.

Let $\varphi(\zeta,X)$ denote the spherical function with spectral
parameter $\zeta\in\mathfrak a^*_{\mathbb C}$, evaluated at
$X\in\mathfrak a$ and normalized by $\varphi(\zeta,0)=1$.
We use the convention in which the tempered spectral parameters are real,
and put $\varphi_0(X)=\varphi(0,X)$.   

For $\zeta\in\mathfrak a^*_{\mathbb C}$ and $X\in \mathfrak a$, define
\begin{equation}\label{eq:spherical-root-majorant}
 \mathcal R_\Phi(\zeta,X)
 :=\sum_{w\in W}\prod_{\alpha\in\Phi^+}
 \left(1+|\alpha(X)|\,
 |\langle w\zeta,\alpha^\vee\rangle|\right)^{-1},
\end{equation}
and put
\begin{equation}\label{eq:complex-support-function}
 h_\Phi(\eta,X):=
 \max_{w\in W}\{-\langle w\eta,X\rangle\}.
\end{equation}
Thus $e^{h_\Phi(\eta,X)}$ is the largest absolute value of the
exponential terms $e^{i\langle w\zeta,X\rangle}$ over $w\in W$.

For real $\lambda\in\mathfrak a^*$, the explicit formula for
spherical functions on complex groups shows that
Theorem~\ref{thm:linear} is equivalent to the pointwise bound
\begin{equation}\label{eq:spherical-bound}
 |\varphi(\lambda,X)|
 \leq C_G\varphi_0(X)\mathcal R_\Phi(\lambda,X).
\end{equation}
See Section~\ref{sec:real-and-spherical-consequences} for details.

An equivalent form of \eqref{eq:spherical-bound} was also
recently established independently by Brumley, Marshall, Matz, and
Peterson \cite[Theorem~2.5]{BrumleyMarshallMatzPeterson2026},
so their result also implies Theorem~\ref{thm:linear}. 
They obtain the estimate by scaling a bound for spherical functions
on general real semisimple groups, proved using asymptotic expansions
along Levi subgroups, induction, and the maximum principle in the
spectral parameter. 
We give a very different, more algebraic proof of
Theorem~\ref{thm:linear}. We work directly with Weyl alternants,
decomposing them into blocks indexed by parabolic double cosets
and applying the relative BGG--Demazure identity to each block.
Moreover, Theorem~\ref{thm:jets} also yields the following bounds for mixed
radial derivatives.

\begin{theorem}[Radial derivative bounds for spherical functions]
\label{cor:spherical-derivatives}
For every integer $k\geq0$, there exists $C_{G,k}>0$ such that
\begin{equation}\label{eq:spherical-derivative-bound}
\begin{aligned}
 \left|D_{\xi_1}\cdots D_{\xi_k}
       \varphi(\zeta,\cdot)(X)\right|\leq C_{G,k}(1+\|\zeta\|)^k
 \varphi_0(X)e^{h_\Phi(\operatorname{Im}\zeta,X)}
 \mathcal R_\Phi(\zeta,X)
\end{aligned}
\end{equation}
for all $\zeta\in\mathfrak a^*_{\mathbb C}$, $X\in\mathfrak a$,
and unit vectors $\xi_1,\ldots,\xi_k\in\mathfrak a$.
\end{theorem}

In the setting of Theorem~\ref{cor:spherical-derivatives},
Cowling and Nevo \cite[Theorem~1.2]{CowlingNevo2001}
estimate derivatives of every order along fixed regular rays
$X=tH$. Their theorem expresses the oscillatory gain through
the single factor $(1+t\|\zeta\|)^{-\gamma_\Phi}$.
They also obtain stronger decay when the spectral direction
stays away from root hyperplanes
\cite[Corollary~2.5]{CowlingNevo2001}.
Our Theorem~\ref{cor:spherical-derivatives} refines both
estimates by retaining the individual spatial and spectral
root factors uniformly across root hyperplanes and by
allowing arbitrary mixed radial derivatives.
See Section~\ref{subsec: cowling-nevo} for further discussion.

The rootwise form of \eqref{eq:spherical-bound} admits a natural
interpretation in terms of stationary phase through the Harish-Chandra integral
formula for spherical functions \cite{HarishChandra1958}. For regular spatial and spectral parameters, the
critical manifolds in the integral formula are indexed by the Weyl group,
and the transverse Hessian decomposes into root-space blocks. Since
every restricted root of a complex group has multiplicity two,
stationary phase predicts one inverse factor for each positive root.
The difficulty in obtaining a uniform result is that the critical manifold
changes and may become higher-dimensional when either parameter reaches
a root hyperplane.

For spherical functions on general symmetric spaces of noncompact type,
Duistermaat, Kolk, and Varadarajan obtained stationary phase asymptotics
\cite[Theorem~9.1]{DuistermaatKolkVaradarajan1983} and proved similar rootwise bounds that are uniform for all real spectral parameters when the
spatial variable ranges over a fixed compact set; these bounds are sharp
when the compact set is equisingular but generally become coarser otherwise
\cite[Theorem~11.1]{DuistermaatKolkVaradarajan1983}.
Marshall proved a rootwise bound uniformly as the
spatial variable ranges over a fixed compact set, under the assumption
that the spectral direction remains in a compact subset of the regular
set
\cite[Theorem~1.3]{Marshall2016}. 
On the other hand, Blomer and Pohl
\cite[Theorem~2]{BlomerPohl2016} and Matz and Templier
\cite[Proposition~8.2]{MatzTemp2021} established bounds for general
semisimple groups that are uniform on bounded spatial sets and
for all real spectral parameters, but do not capture the full
rootwise decay. For the rank-two group $\mathrm{SL}(3,\mathbb R)$, the first author
obtained a rootwise estimate on bounded spatial sets that remains
uniform as either the spatial variable or the spectral parameter
approaches its corresponding singular set
\cite[Theorem~1.1]{Li2025}. For complex semisimple groups, \eqref{eq:spherical-bound} and \eqref{eq:spherical-derivative-bound} refine the preceding upper bounds by removing the spatial and spectral restrictions and giving the rootwise decay predicted by stationary phase for spherical functions and their mixed radial derivatives.

\subsection{Characters of compact Lie groups}
\label{subsec:intro-characters}

Let $U$ be a compact connected semisimple Lie group. Under the
identification $U\simeq (U\times U)/\operatorname{diag}U$, the irreducible characters of $U$, when normalized to equal $1$
at the identity, are precisely the zonal
spherical functions of this compact symmetric space.

For spherical functions on compact symmetric spaces, Clerc obtained stationary phase asymptotics with the spatial variable and the spectral direction ranging over compact subsets of their respective regular sets \cite[Theorem~3.4]{Clerc1988}. Under the same restriction on the
spectral direction, Marshall proved a rootwise bound when the
spatial variable lies near the base point
\cite[Theorem~1.6]{Marshall2016}.

We now present the analogue of
Theorem~\ref{cor:spherical-derivatives} for characters of compact groups.
In this setting, our bounds remove the preceding restrictions
on the spatial variable and the spectral direction: they hold
on the entire maximal torus, uniformly over all dominant
highest weights. They also apply to arbitrary mixed radial
derivatives. We first record the periodic version of the alternant estimate
and then state the corresponding bounds for characters.

For integral weights, periodicity in the spatial variable replaces distance to a linear root hyperplane by
distance to the corresponding affine root hyperplanes.  Let
$T$ be a maximal torus of $U$, let $\Phi$ be its root system, and let
$\mathfrak t$ be the Lie algebra of $T$.
For $\alpha\in\Phi$ put
\[
 d_\alpha(H)=\abs{\sin\frac{\alpha(H)}2},\qquad H\in\mathfrak t.
\]
A weight $\lambda\in\mathfrak t^*$ is called regular integral if
$\langle\lambda,\alpha^\vee\rangle\in\mathbb Z\setminus\{0\}$ for every
root $\alpha$.
The following result is a periodic analogue of
Theorem~\ref{thm:linear} and follows from it by an argument of localization relative to
the affine Weyl arrangement. It generalizes the type \(A_2\) estimate
proved in \cite[Theorem~1.1]{ZhangSU3}.

\begin{theorem}[Periodic Weyl alternant estimate]
\label{thm:periodic}
There is a constant $C_\Phi$ such that, for every regular integral weight $\lambda\in\mathfrak{t}^*$ and every $H\in\mathfrak t$,
\begin{equation}\label{eq:periodic-main}
 \abs{A_\lambda^\Phi(H)}
 \le C_\Phi\sum_{w\in W}
 \prod_{\alpha\in\Phi^+}
 \min\!\left\{1,\,
   d_\alpha(H)
   \abs{\ip{w\lambda}{\alpha^\vee}}
 \right\}.
\end{equation}
\end{theorem}

A weight $\mu$ in the character lattice of $T$ is dominant if
$\langle\mu,\alpha^\vee\rangle\geq0$ for every simple root $\alpha$.
Each such $\mu$ is the highest weight of an irreducible representation
of $U$; let $\chi_\mu$ denote its character.  By the Weyl character formula~\cite[Satz~5, p.~389]{Weyl1926}, 
\[
\chi_\mu(\exp H)=\frac{A^\Phi_{\mu+\rho}(H)}{A^\Phi_{\rho}(H)},
\]
where $\rho$ denotes the half-sum of positive roots. The absolute value
of the Weyl denominator is
\[
|A^\Phi_\rho(H)|
=
2^{|\Phi^+|}
\prod_{\alpha\in\Phi^+}
d_\alpha(H).
\]
Dividing \eqref{eq:periodic-main} by the Weyl denominator and using
$\min\{1,t\}\asymp t/(1+t)$ for $t\geq 0$ yields the $k=0$ case of the character
estimate below. 
More generally, applying the same localization argument to
Theorem~\ref{thm:jets} and using the Weyl character formula yields
the bounds for all mixed radial derivatives.

\begin{theorem}[Radial derivative bounds for characters]
\label{cor:character}
For every integer $k\geq0$, there is a constant $C_{U,k}$
such that, for every dominant weight $\mu$ in the character
lattice of $T$, every $H\in\mathfrak t$, and all unit
vectors $\xi_1,\ldots,\xi_k\in\mathfrak t$,
\begin{align}
 \left|
 D_{\xi_1}\cdots D_{\xi_k}
 (\chi_\mu\circ\exp)(H)
 \right|
 \le C_{U,k}\|\mu+\rho\|^k\sum_{w\in W}
 \prod_{\alpha\in\Phi^+}
 \frac{\abs{\ip{w(\mu+\rho)}{\alpha^\vee}}}{1+ d_\alpha(H)\abs{\ip{w(\mu+\rho)}{\alpha^\vee}}}.
\label{eq:character-main}
\end{align}
\end{theorem}

For $k=0$, \eqref{eq:character-main} refines earlier character
bounds of Hare \cite{Hare1998,HareWilsonYee2000}; see
Remark~\ref{rem:character-comparison}.
Possible applications to sharp $L^p$ bounds for characters are
discussed in Remark~\ref{rem:character-Lp}.

\subsection{Proof strategy and organization}
\label{subsec:intro-strategy}

The proof of Theorem~\ref{thm:jets} builds on the second
author's earlier work on $\mathrm{SU}(3)$ \cite{ZhangSU3}.
In that work, the Weyl alternant is decomposed into right-coset
blocks adapted to the spatial singularity and then regrouped
according to the spectral singularity. This amounts to a parabolic double-coset decomposition,
which is central to the approach we develop here in arbitrary rank.

After reducing to regular dominant $\lambda$ and dominant $X$,
a gap among the simple-coroot spectral coordinates determines
a spectral parabolic subgroup $Q$. Its associated root
subsystem $\Phi_Q$ has smaller rank and provides the input
for the induction. On the spatial side, the uncertainty
principle suggests the reciprocal scale $\|\lambda\|^{-1}$:
the simple roots satisfying
$\alpha_i(X)\leq\|\lambda\|^{-1}$ determine a spatial
parabolic subgroup $P$. Then we decompose $W$ into double cosets $PuQ$, choosing $u$ of 
minimal length in each double coset. 

As a crucial algebraic step, we develop the relative
BGG--Demazure identity \eqref{eq:relative-BGG} and apply it
to each double-coset block of the Weyl alternant. It extracts the root product associated
with $P$ from each alternant block and leaves a
divided-difference operator acting on an exponential phase
times a lower-rank alternant quotient for $\Phi_Q$; see
Lemma~\ref{lem:exact-block}.

The integral representation of divided-difference operators then provides the analytic link to the
inductive hypothesis: using the fundamental theorem of calculus,
we convert the relative divided-difference operator into an
integral of directional derivatives. Derivatives falling on the
exponential factor contribute powers of $\|\lambda\|$,
while those falling on the lower-rank alternant quotient are
controlled by the inductive hypothesis. 
The minimality of each double-coset representative $u$ gives, in particular, the precise
derivative count in the relative BGG--Demazure identity
needed for the rootwise estimates. 
This proves a local derivative estimate, stated in
Proposition~\ref{prop:local-jets}, which is more convenient
to establish first. Theorem~\ref{thm:jets} then follows
by dividing by the remaining root product and applying
Leibniz's rule.

The paper is organized as follows.
Section~2 records the required facts about parabolic double
cosets. Section~3 develops the relative BGG--Demazure
identity and the integral representation of
divided-difference operators. Section~4 derives the exact
block formula, and Section~5 proves the local derivative
estimate by induction on the rank.
Section~\ref{sec:real-and-spherical-consequences} proves
Theorem~\ref{thm:jets}, derives the estimates for spherical
functions with complex spectral parameters, and compares
them with the Cowling--Nevo estimates. 
Section~7 uses localization near affine root hyperplanes
to prove the periodic alternant estimate and derive the character estimates.

\subsection*{Notation}
For nonnegative quantities $A$ and $B$, we write
$A\lesssim_{\Phi,k}B$ if $A\leq C_{\Phi,k}B$ for some
positive constant depending on $\Phi$ and $k$.
We write $A\asymp_{\Phi,k}B$ if both
$A\lesssim_{\Phi,k}B$ and $B\lesssim_{\Phi,k}A$ hold. Other subscripts similarly indicate 
the dependence of the implicit constants. Constants denoted by $C$ or $c$, with appropriate 
subscripts, may change from line to line.  

\subsection*{Acknowledgments}

We would like to thank Simon Marshall for helpful discussions about the results on spherical functions in \cite{BrumleyMarshallMatzPeterson2026}.  Y. Zhang gratefully acknowledges the hospitality of both the University of Cincinnati and the University of Mississippi during his work on this project. This project is partially supported by the National Key R\&D Program of China under Grant No.~2022YFA1006700. X. Li is also supported by the National Natural Science Foundation of China under Grant No.~12501039.

\subsection*{AI Declaration}

This project grew out of the second author's earlier work on $\mathrm{SU}(3)$ \cite{ZhangSU3}. The authors used GPT-5.6 Sol to help extend the descent strategy of \cite{ZhangSU3} to arbitrary rank. The AI model suggested using an iterated integral representation of divided-difference operators, which helped the authors set up the rank induction in Proposition~\ref{prop:local-jets}. It also suggested using the maximum principle to extend the estimates from real to complex spectral parameters, an argument that appears in the proof of Proposition~\ref{prop:complex-normalized-jets} and gives the corresponding extension in Theorem~\ref{cor:spherical-derivatives}. The authors also used GPT-5.6 Sol and GPT-6 Astra to check mathematical arguments and revise the wording of the manuscript. The authors take full responsibility for the contents of the manuscript.

\section{Parabolic double cosets of the Weyl group}
\label{sec: double cosets}
We first recall the standard structure of one-sided parabolic cosets of the Weyl group and their minimal representatives. We then pass to parabolic double cosets and record the existence of minimal representatives, the parabolic structure of the relevant subgroup intersections, and the factorization of each double coset used later in decomposing the Weyl alternant.

\subsection{One-sided parabolic cosets}
We recall the standard facts about one-sided parabolic cosets. Let $\Phi$ be a reduced crystallographic root system, let $W$ be the associated Weyl group, and let $\Delta$ be the chosen set of simple roots. For $I\subset\Delta$,
let
\(
W_I=\langle s_\alpha:\alpha\in I\rangle
\)
be the corresponding standard parabolic subgroup of $W$, let
$
\Phi_I=\Phi\cap\operatorname{span}_{\mathbb R}(I)$ denote the corresponding root subsystem, and let $\Phi_I^+=\Phi_I\cap\Phi^+$
denote the positive subsystem.
Let \(\ell:W\to\mathbb Z_{\geq0}\) denote the length function with respect to the simple reflections corresponding to \(\Delta\). 
Define
\[
 W^I
 :=
 \{x\in W:\ell(xs_\alpha)>\ell(x)
                  \text{ for every }\alpha\in I\},
\]
and
\[
 {}^I W
 :=
 \{x\in W:\ell(s_\alpha x)>\ell(x)
                  \text{ for every }\alpha\in I\}.
\]
By the standard length--root criterion
\cite[\S 1.7]{Humphreys},
\begin{align}\label{eq: length-root}
x\in W^I
\quad\Longleftrightarrow\quad
x\Phi_I^+\subset\Phi^+,
\qquad
x\in{}^I W
\quad\Longleftrightarrow\quad
x^{-1}\Phi_I^+\subset\Phi^+.
\end{align}
The following lemma is standard;
see \cite[\S 1.10]{Humphreys}.

\begin{lemma}[One-sided parabolic decomposition]\label{lem:one-sided-cosets}
\leavevmode\par
\begin{enumerate}

\item Every $w\in W$ has unique factorizations
\[
 w=xv,
 \qquad x\in W^I,\quad v\in W_I,
\]
and
\[
 w=py,
 \qquad p\in W_I,\quad y\in{}^I W,
\]
and these factorizations are length-additive:
\[
 \ell(xv)=\ell(x)+\ell(v),
 \qquad
 \ell(py)=\ell(p)+\ell(y).
\]

\item Every left coset $wW_I$ has a unique minimal-length element
$x\in W^I$ and a unique maximal-length element $xw_I$, where $w_I$
is the longest element of $W_I$.  
In particular, if $w_0$ denotes the longest element in $W$, then $w_0=xw_I$ for some $x\in W^I$.
Similarly, every right coset
$W_Iw$ has a unique minimal-length element $y\in{}^I W$ and a
unique maximal-length element $w_Iy$, and $w_0=w_Iy$ for some $y\in {}^I W$. 
\end{enumerate}
\end{lemma}

\subsection{Parabolic double cosets}
We will group the terms in the Weyl alternant according to parabolic double
cosets and use the resulting double-coset sums as the basic units for exhibiting cancellation. This construction reformulates and generalizes to arbitrary rank the right-coset
regrouping used in \cite{ZhangSU3}. We will use the
standard structure of these double cosets, summarized in
the following lemma.  Its ingredients are standard; see
\cite[Chapter~IV, \S1, Exercise~3]{BourbakiLie46} or
\cite[Proposition~2.7 and Corollary~2.8]{BilleyEtAl}.  Since neither
reference states the precise formulation below in full, we include a
proof for completeness.

\begin{lemma}[Structure of parabolic double cosets]
\label{lem:double-cosets}
For $I,J\subset\Delta$, let \(P=W_I\) and \(Q=W_J\) be standard parabolic subgroups of \(W\).
Every double coset \(PuQ\) contains 
an element \(u\) of minimal
length. For this element, set
\[
 R=P\cap uQu^{-1},
 \qquad
 R'=Q\cap u^{-1}Pu=u^{-1}Ru.
\]
Then the following statements hold.
\begin{enumerate}
    \item
    The groups \(R\) and \(R'\) are standard parabolic subgroups of
    \(P\) and \(Q\), respectively.

     \item
    The corresponding root systems satisfy
    \begin{equation}
    \label{eq:intersection-root-systems}
        \Phi_R=\Phi_P\cap u\Phi_Q,
        \qquad
        \Phi_{R'}=\Phi_Q\cap u^{-1}\Phi_P,
        \qquad
        u\Phi_{R'}=\Phi_R.
    \end{equation}
    Their positive systems satisfy
    \begin{equation}
    \label{eq:intersection-positive-systems}
        \Phi_R^+
        =\Phi_P^+\cap u\Phi_Q^+,
        \qquad
        \Phi_{R'}^+
        =\Phi_Q^+\cap u^{-1}\Phi_P^+,
        \qquad
        u\Phi_{R'}^+=\Phi_R^+.
    \end{equation}

    \item 
We have the 
disjoint decomposition
\begin{equation}\label{eq:four-root-classes}
\begin{aligned}
 \Phi^+=
 (\Phi_P^+\setminus\Phi_R^+)
 \sqcup\Phi_R^+ 
 \sqcup u(\Phi_Q^+\setminus\Phi_{R'}^+)
 \sqcup
 \bigl(\Phi^+\setminus(\Phi_P^+\cup u\Phi_Q^+)\bigr).
\end{aligned}
\end{equation}
    \item
    Let \(P^R\) denote the set of minimal representatives of the left
    cosets \(P/R\). Multiplication induces a bijection
    \[
        P^R\times Q\longrightarrow PuQ,
        \qquad
        (p,q)\longmapsto puq.
    \]
\end{enumerate}
\end{lemma}

\begin{proof}
Since \(W\) is finite, the double coset contains an element \(u\) of
minimal length. Its minimality implies $ \ell(s_\alpha u)>\ell(u)$ for all $\alpha\in I$, and $\ell(us_\beta)>\ell(u)$ for all $\beta\in J$.
Thus $u\in{}^I W\cap W^J$, and by \eqref{eq: length-root}, we have 
\begin{equation}
\label{eq:u-positivity}
 u^{-1}\Phi_P^+\subset\Phi^+,
 \qquad
 u\Phi_Q^+\subset\Phi^+.
\end{equation}

By Solomon's intersection lemma \cite[Lemma~2]{Sol76},
there exist subsets $I_0\subset I$ and $J_0\subset J$ such that
\[
P\cap uQu^{-1}=W_{I_0},
\qquad
Q\cap u^{-1}Pu=W_{J_0}.
\]
More explicitly, one may take
\[
I_0=\{\alpha\in I:u^{-1}\alpha\in J\},
\qquad
J_0=\{\beta\in J:u\beta\in I\}.
\]
Consequently, $R=W_{I_0}$ and $R'=W_{J_0}$ are standard
parabolic subgroups of $P$ and $Q$, respectively.

We next identify their root systems.
Since \(u^{-1}s_\alpha u=s_{u^{-1}\alpha}\), and since a root
reflection belongs to a parabolic subgroup precisely when its
root belongs to the corresponding root subsystem, we have
\[
 \Phi_R
 =\{\alpha\in\Phi_P:u^{-1}\alpha\in\Phi_Q\}
 =\Phi_P\cap u\Phi_Q.
\]
Similarly,
$\Phi_{R'}=\Phi_Q\cap u^{-1}\Phi_P$.
This proves \eqref{eq:intersection-root-systems}. 

We next identify the positive systems.  
We first prove $\Phi_R^+=\Phi_P^+\cap u\Phi_Q^+$.
If
$\alpha\in\Phi_P^+\cap u\Phi_Q^+$, 
then \(\alpha\in\Phi_P\cap u\Phi_Q=\Phi_R\), and \(\alpha\) is
positive. 
This proves
$\Phi_P^+\cap u\Phi_Q^+\subset\Phi_R^+$.
Conversely, let \(\alpha\in\Phi_R^+\). Since $R$ is a standard parabolic subgroup of $P$, we have $\alpha\in\Phi_P^+$. Therefore the first inclusion in
\eqref{eq:u-positivity} gives $u^{-1}\alpha\in\Phi^+$.
Moreover, the identity $\Phi_R=\Phi_P\cap u\Phi_Q$
shows that
$u^{-1}\alpha\in\Phi_Q$, and therefore $u^{-1}\alpha\in\Phi_Q\cap\Phi^+=\Phi_Q^+$. We conclude that 
$\alpha\in u\Phi_Q^+$, and $\Phi_R^+\subset\Phi_P^+\cap u\Phi_Q^+$.
We have proved $\Phi_R^+=\Phi_P^+\cap u\Phi_Q^+$; the other identity 
$\Phi_{R'}^+
    =\Phi_Q^+\cap u^{-1}\Phi_P^+$ can be proved in a similar way, and we have thus finished the proof of \eqref{eq:intersection-positive-systems}.

We next prove the disjoint decomposition of \(\Phi^+\).  By
\eqref{eq:u-positivity}, \(u\Phi_Q^+\subset\Phi^+\), while
\(\Phi_P^+\subset\Phi^+\) by definition.  These two subsets determine
the disjoint partition of \(\Phi^+\) into
\[
 \Phi_P^+\setminus u\Phi_Q^+,\qquad
 \Phi_P^+\cap u\Phi_Q^+,\qquad
 u\Phi_Q^+\setminus\Phi_P^+,\qquad
 \Phi^+\setminus(\Phi_P^+\cup u\Phi_Q^+);
\]
see Figure~\ref{fig:four-root-classes}.  By
\eqref{eq:intersection-positive-systems}, $ \Phi_P^+\cap u\Phi_Q^+=\Phi_R^+=u\Phi_{R'}^+$. 
Consequently,
\[
\begin{aligned}
 \Phi_P^+\setminus u\Phi_Q^+
   &=\Phi_P^+\setminus\Phi_R^+,\\
 u\Phi_Q^+\setminus\Phi_P^+
   &=u\Phi_Q^+\setminus\Phi_R^+
     =u\bigl(\Phi_Q^+\setminus\Phi_{R'}^+\bigr).
\end{aligned}
\]
The disjoint decomposition \eqref{eq:four-root-classes} follows.

We finally derive the double-coset factorization. Let \(a\in P\) and \(b\in Q\). Applying the one-sided coset decomposition of \(P\) with respect to
\(R\), as in Lemma~\ref{lem:one-sided-cosets}, write uniquely
\[
 a=pr,
 \qquad
 p\in P^R,\quad r\in R,
\]
with \(  \ell(a)=\ell(p)+\ell(r)\). 
Since \(r\in R\) and \(R'=u^{-1}Ru\subset Q\), the element $r'=u^{-1}ru$ 
belongs to \(R'\subset Q\). Therefore
\[
 aub=prub
     =pu(u^{-1}ru)b
     =pu(r'b).
\]
This proves that every element of \(PuQ\) can be written as
$puq$, where $p\in P^R$ and $q\in Q$.
For uniqueness, suppose
\[
 puq=p'uq',
 \qquad
 p,p'\in P^R,\quad q,q'\in Q.
\]
Then
\[
 p'^{-1}p=u(q'q^{-1})u^{-1}.
\]
The left-hand side belongs to \(P\), while the right-hand side belongs
to \(uQu^{-1}\). Hence
\[
 p'^{-1}p\in P\cap uQu^{-1}=R.
\]
Thus \(p\) and \(p'\) belong to the same left \(R\)-coset. As \(p\) and \(p'\) are both minimal-length representatives of the same
left \(R\)-coset in \(P\), by the uniqueness assertion in
Lemma~\ref{lem:one-sided-cosets}, we have \(p=p'\). 
It then follows from \(puq=puq'\) that \(q=q'\). Therefore, $(p,q)\mapsto puq$ is a bijection from $P^R\times Q$ onto $PuQ$.
\end{proof}

\begin{figure}[t]
\centering
\begin{tikzpicture}[font=\small]
  \draw[rounded corners=2pt]
    (-4,-2.2) rectangle (4,2.2);
  \node[anchor=north west] at (-3.85,2.05) {$\Phi^+$};

  \filldraw[
    fill=blue!15,
    fill opacity=.5,
    draw=blue!60
  ] (-1.1,0) ellipse (2.35 and 1.4);

  \filldraw[
    fill=red!15,
    fill opacity=.5,
    draw=red!60
  ] (1.1,0) ellipse (2.35 and 1.4);

  \node at (-2.1,0.95) {$\Phi_P^+$};
  \node at (2.1,0.95) {$u\Phi_Q^+$};

  \node[font=\bfseries] at (-2.15,0) {1};
  \node[font=\bfseries] at (0,0) {2};
  \node[font=\bfseries] at (2.15,0) {3};
  \node[font=\bfseries] at (3.45,-1.7) {4};
\end{tikzpicture}

\medskip

\begin{tabular}{@{}c@{\qquad}l@{}}
\textbf{1} & \(\Phi_P^+\setminus\Phi_R^+\), \\[1mm]
\textbf{2} & \(\Phi_R^+=u\Phi_{R'}^+\), \\[1mm]
\textbf{3} & \(u\bigl(\Phi_Q^+\setminus\Phi_{R'}^+\bigr)\), \\[1mm]
\textbf{4} &
\(\Phi^+\setminus\bigl(\Phi_P^+\cup u\Phi_Q^+\bigr)\).
\end{tabular}

\caption{The four root classes in
\eqref{eq:four-root-classes}.}
\label{fig:four-root-classes}
\end{figure}
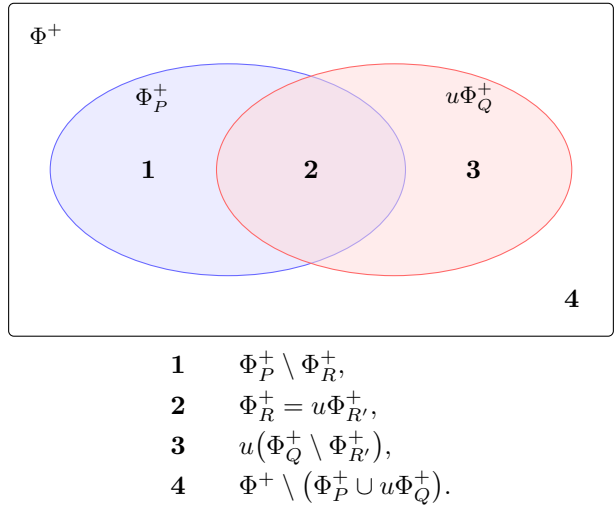

\begin{remark}
As in the one-sided case, one may prove that for the minimal representative \(u\) of the double coset \(PuQ\), the factorization
$$
w=puq,
\qquad
p\in P^R,\quad q\in Q,
$$
is length-additive:
$$
\ell(puq)=\ell(p)+\ell(u)+\ell(q).
$$
Consequently, \(u\) is the unique minimal-length element of \(PuQ\).
However, neither the preceding length-additivity property nor the uniqueness of the minimal-length representative will be used below.

\end{remark}

\section{Relative BGG--Demazure identity}\label{sec: BGG}
We develop a relative BGG--Demazure identity for an inclusion
\(R\subset P\) of standard parabolic subgroups.  It expresses the alternating
sum of the translates of an $R$-skew function 
over the minimal representatives of the left cosets \(P/R\) in
terms of the divided-difference operator associated with the relative
longest element \(w_P^R\). These divided-difference operators originate in the work of
Bernstein--Gel'fand--Gel'fand \cite{BGG1973} and Demazure
\cite{Demazure1974}. We refer to \cite[Chapter~IV]{Hil82}
for a detailed exposition. The standard BGG--Demazure identity was used in the
second author's earlier work \cite{ZhangStrichartz} to
estimate differences of characters in the study of
Schr\"odinger flow on compact Lie groups.

Let $K=W_I$ be a standard parabolic subgroup of $W$,
where $I\subset\Delta$.
For a simple root $\alpha\in I$ and $\alpha(Y)\neq0$,
we define the
divided-difference operator on smooth functions by
\[
 \partial_{s_\alpha} f(Y)
 =\frac{f(Y)-f(s_\alpha Y)}{\alpha(Y)},
 \qquad f\in C^\infty(V).
\]
Since $Y-s_\alpha Y=\alpha(Y)\alpha^\vee$, the
fundamental theorem of calculus gives
\begin{equation}\label{eq: FTC-one}
 \partial_{s_\alpha}f(Y)
  =\int_0^1
    [D_{\alpha^\vee}f]
      \bigl(s_\alpha Y+t(Y-s_\alpha Y)\bigr)\dd t.
\end{equation}
The right-hand side defines a smooth extension of
$\partial_{s_\alpha}f$ to all $Y\in V$, including those
satisfying $\alpha(Y)=0$. Thus $\partial_{s_\alpha}$
maps $C^\infty(V)$ to itself.

For $Y\in V$, $f\in C^\infty(V)$, define 
\[
 \pi_K(Y)=\prod_{\alpha\in\Phi_K^+}\alpha(Y),
 \qquad
 \Alt_K f=\sum_{k\in K}\det(k)k(f),
\]
where $k(f)(Y)=f(k^{-1}Y)$. 
We record the BGG--Demazure identity for smooth functions.

\begin{lemma}[BGG--Demazure identity]\label{lem: BGG}
Let $w_K$ be the longest element of $K$.
Given a reduced expression
$w_K=s_{\alpha_1}\cdots s_{\alpha_n}$,
where $\alpha_1,\ldots,\alpha_n\in I$, define 
$$\partial_{w_K}=\partial_{s_{\alpha_1}}\cdots\partial_{s_{\alpha_n}}:C^\infty(V)\to C^\infty(V).$$
Then for all $f\in C^\infty(V)$, we have the BGG--Demazure identity  
\begin{equation}\label{eq: BGG}
 \partial_{w_K}f=\frac{\Alt_K f}{\pi_K}.
\end{equation}
In particular, the definition of $\partial_{w_K}:C^\infty(V)\to C^\infty(V)$ does not depend on the chosen reduced expression of $w_K$. 
\end{lemma}

\begin{proof}
The statements are true if $C^\infty(V)$ is replaced by the space
$S(V)$ of polynomial functions on $V$; see \cite[Chapter~IV, Proposition~1.6]{Hil82}. Let $f\in C^\infty(V)$ and let $Y\in V$
be regular. For any fixed reduced expression of $w_K$, the value
$(\partial_{w_K}f)(Y)$ defined using this expression depends only on
the values of $f$ on the Weyl group orbit of $Y$.
Since the $K$-orbit of $Y$ is finite, there exists a polynomial
$\widetilde{f}\in S(V)$ that agrees with $f$ on the $K$-orbit of $Y$.
Likewise, $(\Alt_K f)(Y)$ depends only on the values of $f$ on
the $K$-orbit of $Y$. Hence
\[
(\partial_{w_K}f)(Y)
=(\partial_{w_K}\widetilde{f})(Y)
=\frac{(\Alt_K\widetilde{f})(Y)}{\pi_K(Y)}
=\frac{(\Alt_K f)(Y)}{\pi_K(Y)}.
\]
Thus \eqref{eq: BGG} holds on the dense open set of regular points,
and there the resulting value is independent of the chosen reduced
expression of $w_K$.

By \eqref{eq: FTC-one}, the operator associated with each reduced
expression maps $C^\infty(V)$ to $C^\infty(V)$ and is defined at every
$Y\in V$. Since the operators associated with any two reduced
expressions agree on the dense set of regular points, they agree on
all of $V$. Consequently, $\Alt_K f/\pi_K$, initially defined at
regular points, has a smooth extension to all of $V$, equal to
$\partial_{w_K}f$. This proves \eqref{eq: BGG} and the independence of
the chosen reduced expression.
\end{proof}

More generally, for any $w\in K$ with reduced expression
$w=s_{\alpha_1}\cdots s_{\alpha_n}$ in the simple reflections
associated with $\Phi_K^+$, define
\[
 \partial_w
 :=
 \partial_{s_{\alpha_1}}\cdots\partial_{s_{\alpha_n}}
 :C^\infty(V)\longrightarrow C^\infty(V),
 \qquad
 \partial_e:=\mathrm{Id}.
\]
This definition is independent of the chosen reduced expression.
Indeed,  using Lemma~\ref{lem: BGG}, the rank-two reduction argument in
\cite[Chapter~IV, Proposition~1.7]{Hil82} carries over verbatim
from $S(V)$ to $C^\infty(V)$.

Let $P$ be a standard parabolic subgroup of $W$, and let $R$
be a standard parabolic subgroup of $P$.  If $w_R,w_P$ are the
longest elements, then by Lemma~\ref{lem:one-sided-cosets},
there exists $w_P^R\in P^R$ such that 
\begin{align}\label{eq:wpR}
 w_P=w_P^R w_R,\qquad
 n:=\ell(w_P^R)=\ell(w_P)-\ell(w_R)=|\Phi_P^+|-|\Phi_R^+|.
\end{align}
For \(F\in C^\infty(V)\), we say that \(F\) is \(R\)-invariant if
\(r(F)=F\) for every \(r\in R\), and \(R\)-skew if
\( r(F)=\det(r)F\) for every $r\in R$.
We now establish a 
relative version of the BGG--Demazure identity. 

\begin{lemma}[Relative BGG--Demazure identity]\label{lem:relative-BGG}
Let $G\in C^\infty(V)$ be $R$-invariant. Then 
\begin{equation}\label{eq:relative-BGG}
 \sum_{p\in P^R}\det(p)\,p(\pi_RG)
   =\pi_P\,\partial_{w_P^R}G.
\end{equation}
\end{lemma}

\begin{proof}
As $\pi_R$ is $R$-skew and $G$ is $R$-invariant, $\pi_RG$ is $R$-skew. Applying \eqref{eq: BGG} for $R$, we have
\begin{equation}\label{eq:relative-BGG-1}
   \partial_{w_R}(\pi_RG)
  =\frac{\Alt_R(\pi_RG)}{\pi_R}=|R|G.  
\end{equation}
As the factorization $w_P=w_P^R w_R$ is length-additive, reduced expressions for $w_P^R$ and $w_R$ concatenate to a reduced expression for $w_P$. Hence,
\[
\partial_{w_P}=\partial_{w_P^R}\partial_{w_R}.
\] 
Applying \eqref{eq: BGG} for $P$ and \eqref{eq:relative-BGG-1}, we have
\begin{align}\label{eq:relative-BGG-2}
     \frac{\Alt_P(\pi_RG)}{\pi_P}
 =\partial_{w_P}(\pi_RG)=\partial_{w_P^R}\partial_{w_R}(\pi_RG)=|R|\partial_{w_P^R}G.
\end{align}
On the other hand, using Lemma~\ref{lem:one-sided-cosets} and splitting
$P$ into its left $R$-cosets, we have
\begin{align*}
     \Alt_P(\pi_RG)
    &=\sum_{p\in P^R}\sum_{r\in R}\det(pr)(pr)(\pi_RG).
\end{align*}
Using $\det(pr)=\det(p)\det(r)$ and $(pr)(\pi_RG)=p(r(\pi_RG))=\det(r)p(\pi_RG)$, we then have 
\begin{align*}
     \Alt_P(\pi_RG)
    &=|R|\sum_{p\in P^R}\det(p)p(\pi_RG).
\end{align*}
Combining this identity with \eqref{eq:relative-BGG-2} proves the result.
\end{proof}
For later reference, for a simple root $\alpha$, write 
\begin{align}\label{eq:T-alpha-t}
     T_{\alpha,t}
=s_\alpha+t(\mathrm{Id}-s_\alpha)
=t\,\mathrm{Id}+(1-t)s_\alpha,\qquad 0\le t\le1.
\end{align}
Iterating  \eqref{eq: FTC-one}, we obtain the integral representation of general divided-difference operators. 

\begin{lemma}[Integral representation of divided-difference operators]
\label{lem:integral}
For a Weyl group element \(w\) with reduced expression
\[
 w=s_{\alpha_1}\cdots s_{\alpha_n},
\]
we have
\begin{equation}\label{eq:integral}
 \partial_w f(X)
 =
 \int_{[0,1]^n}
[D_{v_1(\mathbf t)}\cdots D_{v_n(\mathbf t)}
 f]\bigl(Y(\mathbf t,X)\bigr)\,\dd\mathbf t,
\end{equation}
where
$\mathbf{t}=(t_1,\ldots,t_n)$, 
\begin{equation}\label{eq:Y-t}
 Y(\mathbf t,X)
 =
 T_{\alpha_n,t_n}\cdots T_{\alpha_1,t_1}X
\end{equation}
and
\[
 v_j(\mathbf t)
 =
 T_{\alpha_n,t_n}\cdots
 T_{\alpha_{j+1},t_{j+1}}\alpha_j^\vee,
 \qquad j=1,\ldots,n.
\]
\end{lemma}

We record a uniform bound for the displacement of the
evaluation points in \eqref{eq:integral} in terms of the simple-root
values \(\alpha(X)\), for later use. 

\begin{lemma}[Displacement of evaluation points]\label{lem:evaluation}
Let \(W_I\) be a standard parabolic subgroup of \(W\).  For
\(w\in W_I\) and \(X\in V\), every evaluation point
\(Y=Y(\mathbf t,X)\) occurring in the iterated integral representation
\eqref{eq:integral} of \(\partial_w f(X)\) satisfies
\begin{equation}\label{eq:evaluation-general}
 \|Y-X\|
 \leq C_\Phi\max_{\alpha\in I}|\alpha(X)|.
\end{equation}
If $I=\varnothing$, the maximum is understood to be zero. 
\end{lemma}

\begin{proof}
For \(w\in W_I\), choose a reduced expression
\[
 w=s_{\beta_1}\cdots s_{\beta_d},
 \qquad \beta_1,\ldots,\beta_d\in I,
\]
and put
\[
 w_j=s_{\beta_1}\cdots s_{\beta_j},
 \qquad w_0=e.
\]
Telescoping and the reflection formula give
\begin{align*}
 X-wX
 &=\sum_{j=1}^d(w_{j-1} X-w_j X)\\
 &=\sum_{j=1}^d
   w_{j-1}\bigl(X-s_{\beta_j}X\bigr)\\
 &=\sum_{j=1}^d
   \beta_j(X)\,w_{j-1}\beta_j^\vee.
\end{align*}
Since the Weyl group acts orthogonally and
$d=\ell(w)\leq|\Phi_I^+|\leq|\Phi^+|$, 
we obtain
\[
 \|X-wX\|
 \leq C_\Phi\max_{\alpha\in I}|\alpha(X)|.
\]
On the other hand, each operator $ T_{\alpha,t}$ 
is a convex combination of $\mathrm{Id}$  and \(s_\alpha\).  Expanding the
product in \eqref{eq:Y-t} therefore shows that $Y(\mathbf t,X)$ lies in the convex hull of the $W_I$-orbit of $X$.
The preceding estimate now implies \eqref{eq:evaluation-general}.
\end{proof}

\begin{remark}\label{rem: smooth normalized}
Let $Q$ be a parabolic subgroup of $W$. Since $A_\lambda^\Phi$ is
$W$-skew, Lemma~\ref{lem: BGG}, applied with $K=W$, shows that
$A_\lambda^\Phi/\pi_W$ admits a canonical smooth extension to all
of $V$. Since $\Phi_Q\subset\Phi$, the quotient $\pi_W/\pi_Q$ is,
up to an overall sign, a product of linear root factors. Consequently,
\[
\frac{A_\lambda^\Phi}{\pi_Q}
=\frac{\pi_W}{\pi_Q}\frac{A_\lambda^\Phi}{\pi_W}
\]
also admits a canonical smooth extension to all of $V$.
\end{remark}

\begin{remark}[Extension to entire functions]\label{rem:complex-BGG}
The integral formula for divided-difference operators \eqref{eq:integral} and the BGG--Demazure
identity \eqref{eq: BGG} remain valid for entire functions on
$V_{\mathbb C}=V\otimes_{\mathbb R}\mathbb C$, with roots and
Weyl group actions extended complex linearly. Indeed, for $f$ entire on $V_{\mathbb C}$, the
fundamental theorem of calculus gives
\[
 (\partial_{s_\alpha}f)(Z)
 =
 \int_0^1
 D_{\alpha^\vee}f\bigl(
 Z-t\alpha(Z)\alpha^\vee
 \bigr)\,\mathrm dt.
\]
The right-hand side defines an entire function of $Z$ and agrees
with $(f(Z)-f(s_\alpha Z))/\alpha(Z)$ away from the root
hyperplane. Iterating these formulas therefore shows that
$\partial_w$ preserves entire functions for any $w\in W$.

After multiplication by the root product, both sides of the
BGG--Demazure identity are entire. Since an entire function on
$V_{\mathbb C}$ is uniquely determined by its restriction to
$V$, the identity holds throughout $V_{\mathbb C}$. In particular, division
of an entire $W$-skew function by the positive-root product
has a canonical entire extension. 
\end{remark}

\section{The exact double-coset block}
\label{sec: exact block}

Let $P,Q$ be standard parabolic subgroups of $W$ associated to a simple system $\Delta$. 
Decompose the Weyl group into double cosets
\[
 W=\bigsqcup_{u\in\cD}PuQ
\]
using the set $\cD$ of minimal representatives. For \(u\in\cD\), set
\( R=P\cap uQu^{-1}\) and \( R'=u^{-1}Ru\) as in Lemma \ref{lem:double-cosets}. 
Let \(B_{u,\lambda}(X)\) denote the contribution of the double coset \(PuQ\) to
the Weyl alternant, namely, 
\[
 B_{u,\lambda}(X)
 :=
 \sum_{w\in PuQ}\det(w)e^{i\langle w\lambda,X\rangle}.
\]

In the induction below, $P$ and $Q$ will be chosen from the
spatial and spectral data, respectively. The order $PuQ$
reflects these roles. In the root factors
$|\alpha(X)|\,|\langle w\lambda,\alpha^\vee\rangle|$,
right multiplication of $w$ by $q\in Q$ replaces $\lambda$
by $q\lambda$, while left multiplication by $p\in P$
corresponds, after reindexing the roots up to sign, to
replacing $X$ by $p^{-1}X$. Thus the spatial subgroup
appears on the left and the spectral subgroup on the right.

Relative to \(Q\), decompose \(\lambda\in V^*\) orthogonally as
\begin{equation}\label{eq:spectral-split}
 \lambda=\lambda_Q^\perp+\lambda_Q,
 \qquad
 \lambda_Q
 =\operatorname{proj}_{\operatorname{span}_{\mathbb R}(\Phi_Q)}
   \lambda,
 \qquad
 \lambda_Q^\perp\in
 \operatorname{span}_{\mathbb R}(\Phi_Q)^\perp.
\end{equation}
We now rewrite the double-coset block.
\begin{lemma}[Double-coset descent formula]\label{lem:block}
We have
\begin{equation}
 B_{u,\lambda}(X)
 =
 \det(u)\sum_{p\in P^R}\det(p)
 e^{i\ip{\lambda_Q^\perp}{u^{-1}p^{-1}X}}\,
 A_{\lambda_Q}^{\Phi_Q}(u^{-1}p^{-1}X).
\end{equation}

\end{lemma}

\begin{proof}
Applying Lemma \ref{lem:double-cosets}, since \(u\) is the minimal representative of \(PuQ\), every
\(w\in PuQ\) has a unique factorization
\[
 w=puq,
 \qquad
 p\in P^R,\qquad q\in Q.
\]
We may therefore write
\begin{equation*}
 B_{u,\lambda}(X)
 =
 \sum_{p\in P^R}\sum_{q\in Q}
 \det(puq)e^{i\langle puq\lambda,X\rangle}.
\end{equation*}

Applying the spectral decomposition \eqref{eq:spectral-split}, as $Q$ acts trivially on $\operatorname{span}_{\mathbb R}(\Phi_Q)^\perp$, we have 
$q\lambda_Q^\perp=\lambda_Q^\perp$ for all $q\in Q$, and thus 
\[
 q\lambda=\lambda_Q^\perp+q\lambda_Q.
\]
Using Weyl invariance of the Euclidean inner product, we obtain
\begin{align*}
 \langle puq\lambda,X\rangle
 &=
 \langle q\lambda,u^{-1}p^{-1}X\rangle                                      \\
 &=
 \langle \lambda_Q^\perp,u^{-1}p^{-1}X\rangle
 +
 \langle q\lambda_Q,u^{-1}p^{-1}X\rangle.
\end{align*}
Consequently, using the multiplicativity of the determinant, 
we have 
\begin{align*}
 B_{u,\lambda}(X)
 =
 \det(u)
 \sum_{p\in P^R}\det(p)
 e^{i\langle \lambda_Q^\perp,u^{-1}p^{-1}X\rangle}
 \sum_{q\in Q}\det(q)
 e^{i\langle q\lambda_Q,u^{-1}p^{-1}X\rangle}.
\end{align*}
Noting the definition of the Weyl alternant for the subsystem \(\Phi_Q\),
we have finished the proof.
\end{proof}

\begin{remark}[Relation to one-sided descent]
When \(P=\{e\}\) or \(Q=\{e\}\), Lemma~\ref{lem:block} specializes to the corresponding block
identity for a one-sided parabolic coset. These identities are
Weyl-alternant forms of the descent used by Cowling and Nevo for
spherical functions on complex semisimple groups
\cite{CowlingNevo2001}; their counterpart for characters of compact Lie groups appears in
\cite{ZhangFourierRestriction, ZhangCharacterRestriction}. Thus the lemma extends these descent formulas to parabolic
double cosets. 
\end{remark}

Finally, we apply the relative BGG--Demazure identity established in the
preceding section to rewrite each double-coset block once more, this time
in a form suited to the induction-on-rank proof of
Proposition~\ref{prop:local-jets} carried out in the next section.
Define
 \begin{align}
 G_u(Y)
 &=
 e^{i\ip{\lambda_Q^\perp}{u^{-1}Y}}\,
 \frac{A_{\lambda_Q}^{\Phi_Q}(u^{-1}Y)}
      {\pi_{R'}(u^{-1}Y)},\qquad 
  F_u(Y)=
 e^{i\ip{\lambda_Q^\perp}{u^{-1}Y}}\,
 A_{\lambda_Q}^{\Phi_Q}(u^{-1}Y).
 \label{eq:G}
\end{align}

Since \(R'\) is a parabolic subgroup of \(Q\), by Remark \ref{rem: smooth normalized}, 
$
A_{\lambda_Q}^{\Phi_Q}/\pi_{R'}
$
admits a canonical smooth extension. In particular, this implies that \(G_u\in C^\infty(V)\). 

\begin{lemma}\label{lem:exact-block}
\begin{equation}\label{eq:exact-block}
 B_{u,\lambda}(X)
 =
 \det(u)\pi_P(X)\,\partial_{w_P^R}G_u(X).
\end{equation}
\end{lemma}

\begin{proof}
By Lemma~\ref{lem:double-cosets}, we have
\( u\Phi_{R'}^+=\Phi_R^+\). 
Consequently,
\[
 \pi_{R'}(u^{-1}Y)=
 \prod_{\alpha\in \Phi_{R'}^+}\alpha(u^{-1}Y)=
 \prod_{\alpha\in \Phi_{R}^+}\alpha(Y)=
 \pi_R(Y).\]
Thus by \eqref{eq:G}, we have $F_u(Y)=\pi_R(Y)G_u(Y)$.

We claim that \(F_u\) is \(R\)-skew. 
In fact, \(F_u\) is \(uQu^{-1}\)-skew.  It is convenient first to conjugate by \(u\).  Define
\[
\widetilde F_u(Z):=(F_u\circ u)(Z)=F_u(uZ)
=
e^{i\langle\lambda_Q^\perp,Z\rangle}
A_{\lambda_Q}^{\Phi_Q}(Z).
\]
Since \(Q\) fixes \(\lambda_Q^\perp\), the exponential factor is
\(Q\)-invariant, whereas \(A_{\lambda_Q}^{\Phi_Q}\) is \(Q\)-skew.
Hence \(\widetilde F_u\) is \(Q\)-skew.
Conjugating back, we conclude that \(F_u\) is
\(uQu^{-1}\)-skew.  Indeed, if \(q\in Q\), then
\[
\begin{aligned}
F_u\bigl((uqu^{-1})^{-1}Y\bigr)
&=
\widetilde F_u(q^{-1}u^{-1}Y) \\
&=
\det(q)\,\widetilde F_u(u^{-1}Y) \\
&=
\det(uqu^{-1})\,F_u(Y).
\end{aligned}
\]
Since
\(
R=P\cap uQu^{-1}\subset uQu^{-1}
\), 
it follows that \(F_u\) is \(R\)-skew. 

Since \(\pi_R\) is also \(R\)-skew and
\(F_u=\pi_R G_u\), it follows that \(G_u\) is \(R\)-invariant.
By Lemma \ref{lem:block} and \(F_u=\pi_R G_u\), we have
\[
 B_{u,\lambda}(X)
 =
 \det(u)
 \sum_{p\in P^R}\det(p)\,p(\pi_R G_u)(X).
\]
Lemma~\ref{lem:relative-BGG} therefore applies to \(G_u\) and gives
the desired identity.
\end{proof}

\section{A local derivative estimate}
The exact double-coset block identity \eqref{eq:exact-block},
with $G_u$ defined in \eqref{eq:G}, involves a lower-rank
alternant quotient evaluated at points $u^{-1}Y$.
To carry out the induction on the rank, we formulate
Proposition~\ref{prop:local-jets} below as a local derivative estimate
under a smallness condition adapted to these evaluation points. 
We refer to Theorem~\ref{thm:jets} as the global derivative estimate.

The induction requires estimates for all derivative orders at each
lower rank. For a target order $k$, applying the relative
divided-difference operator $\partial_{w_P^R}$ in
\eqref{eq:exact-block} may require lower-rank estimates of orders
up to $k+|\Phi^+|$.
In Section~\ref{sec:real-and-spherical-consequences}, we will deduce
Theorem~\ref{thm:jets} from Proposition~\ref{prop:local-jets}
by dividing by the remaining root product and applying Leibniz's rule.

Let $\Theta$ be a parabolic root subsystem of $\Phi$, that is,
$\Theta=w(\Phi\cap\operatorname{span}_{\mathbb R}(I))$ for some
$w\in W$ and $I\subset\Delta$. Put
\[
 \Theta^+=\Theta\cap\Phi^+,
 \qquad
 \pi_\Theta(X)=\prod_{\alpha\in\Theta^+}\alpha(X).
\]
The quotient \(A_\lambda^\Phi/\pi_\Theta\), initially defined where
\(\pi_\Theta\neq0\), has a canonical smooth extension to all of \(V\); see
Remark~\ref{rem: smooth normalized}.

\begin{proposition}[Local derivative estimate for alternant quotients]
\label{prop:local-jets}
For every integer $k\geq0$, there exist constants
$c_{\Phi,k},C_{\Phi,k}>0$ with the following property.  Let $\Theta\subset\Phi$ be a parabolic root subsystem, and let
$\lambda\in V^*$ and $X\in V$ satisfy
\begin{equation}\label{eq:theta-small}
 \|\lambda\|\,|\alpha(X)|\leq c_{\Phi,k}
 \qquad\text{for all }\alpha\in\Theta.
\end{equation}
Then, for all unit vectors
$\xi_1,\ldots,\xi_k\in V$,
\begin{align}
 \left|
 D_{\xi_1}\cdots D_{\xi_k}
 \left(\frac{A_\lambda^\Phi}{\pi_\Theta}\right)(X)
 \right|
 \notag& \leq
 C_{\Phi,k}
 \sum_{w\in W}
 \sum_{\substack{
 S\subset\Phi^+\setminus\Theta^+\\
 |S|\leq k}}
 \|\lambda\|^{k-|S|}
 \prod_{\alpha\in\Theta^+\cup S}
 |\langle w\lambda,\alpha^\vee\rangle|
 \notag\\[-1mm]
 &\qquad\qquad\times
 \prod_{\alpha\in
 \Phi^+\setminus(\Theta^+\cup S)}
 \min\!\left\{
 1,\,
 |\alpha(X)|
 |\langle w\lambda,\alpha^\vee\rangle|
 \right\}.
\label{eq:jet-estimate}
\end{align}
\end{proposition}

Here $S$ records the ``deleted'' roots in the corresponding
summand, namely those $\alpha\in\Phi^+\setminus\Theta^+$ for
which differentiation may remove the linear factor $\alpha(X)$.
Each such deletion uses one derivative, so $|S|\leq k$, and the
remaining $k-|S|$ derivatives account for the factor
$\|\lambda\|^{k-|S|}$.

\begin{remark}

In fact, Theorem~\ref{thm:jets} gives a stronger form of
\eqref{eq:jet-estimate}. Applying Leibniz's rule to
\[
 \frac{A_\lambda^\Phi}{\pi_\Theta}
 =
 \left(\prod_{\alpha\in\Phi^+\setminus\Theta^+}\alpha\right)
 \frac{A_\lambda^\Phi}{\pi},
\]
we obtain \eqref{eq:jet-estimate} without the smallness condition
\eqref{eq:theta-small} and with the additional factor
\[
 \prod_{\alpha\in\Theta^+\cup S}
 \frac{1}{1+|\alpha(X)|\,|\langle w\lambda,\alpha^\vee\rangle|}
\]
in each summand indexed by $(w,S)$. Under \eqref{eq:theta-small},
the factors for $\alpha\in\Theta^+$ are comparable to one, so the
additional decay comes from the roots in $S$.

The local estimate \eqref{eq:jet-estimate} is better suited to
induction on the rank and suffices to recover the full rootwise
decay in Theorem~\ref{thm:jets}. For the choice of $\Theta$ made in that
proof, dividing by the remaining root product supplies a factor
$|\alpha(X)|^{-1}$ for each $\alpha\in S$.
By \eqref{eq:deleted-root-global-absorption}, each such factor
recovers the corresponding decay at the cost of one power of
$\|\lambda\|$, bringing the total power from $k-|S|$ to $k$.
\end{remark}

We begin the proof of Proposition~\ref{prop:local-jets} by
considering spectral parameters in the root span
$E=\operatorname{span}_{\mathbb R}(\Phi)$.
When $\lambda\in E$, the Weyl alternant and its quotients by
root products depend only on the component of $X$ in $E$.
For general $\lambda$, its component orthogonal to $E$
contributes a common exponential factor, which will be handled
by Leibniz's rule at the end of the proof.

\subsection{First reductions}
\label{subsec: reduction}

We call $X\in V$ dominant if $\alpha(X)\geq0$ for all
$\alpha\in\Delta$. Similarly, we call $\lambda\in V^*$ dominant
if $\langle\lambda,\alpha^\vee\rangle\geq0$ for all
$\alpha\in\Delta$. Every $W$-orbit in $V$ or $V^*$ contains
a dominant point.

For $u\in W$ one has
\begin{align}\label{eq: skew}
 A_{u\lambda}^\Phi(X)=\det(u)A_\lambda^\Phi(X),
 \qquad
 A_\lambda^\Phi(uX)=\det(u)A_\lambda^\Phi(X).
\end{align}

If $\lambda$ is singular, then some reflection $s_\alpha$ fixes it. The first identity above implies that 
$A_\lambda^\Phi\equiv0$, so \eqref{eq:jet-estimate} is immediate. Hence, for
the proof, we may assume that $\lambda$ is regular.

For $k\in\mathbb{Z}_{\geq 0}$, define
\begin{align}
 \mathcal M_{\Phi,\Theta}^{(k)}(\lambda,X)
 :=\sum_{w\in W}
 \sum_{\substack{S\subset\Phi^+\setminus\Theta^+\\|S|\le k}}
 &\norm{\lambda}^{\,k-|S|}
 \prod_{\alpha\in\Theta^+\cup S}\abs{\ip{w\lambda}{\alpha^\vee}}
\notag\\[-1mm]
 &\times
 \prod_{\alpha\in
   \Phi^+\setminus(\Theta^+\cup S)}
\min\{1,\abs{\alpha(X)}\abs{\ip{w\lambda}{\alpha^\vee}}\}.
\label{eq:jet-majorant}
\end{align}
With this notation, the right-hand side of \eqref{eq:jet-estimate} is
$C_{\Phi,k}\mathcal M_{\Phi,\Theta}^{(k)}(\lambda,X)$.  Thus the
conclusion of Proposition~\ref{prop:local-jets} may be written compactly as
\[
 \abs{
 D_{\xi_1}\cdots D_{\xi_k}
 \left(\frac{A_\lambda^\Phi}{\pi_\Theta}\right)(X)}
 \le C_{\Phi,k}\mathcal M_{\Phi,\Theta}^{(k)}(\lambda,X).
\]

The above estimate is covariant under independent Weyl actions on the spectral
and spatial data. More precisely, if \(r,s\in W\), then there exists
a sign \(\varepsilon_{r,s,\Theta}\in\{\pm1\}\) such that
\begin{align}\label{eq: covariance}
D_{s\xi_1}\cdots D_{s\xi_k}
\left(\frac{A_{r\lambda}^\Phi}{\pi_{s\Theta}}\right)(sX)
=
\varepsilon_{r,s,\Theta}
D_{\xi_1}\cdots D_{\xi_k}
\left(\frac{A_\lambda^\Phi}{\pi_\Theta}\right)(X).
\end{align}
The above identity follows from the skew symmetries  \eqref{eq: skew}, 
the identity
\(
\pi_{s\Theta}(sX)=\pm\pi_\Theta(X)\), 
and the chain rule. 
Moreover, the transformation
\(
(\lambda,X,\Theta)
\longmapsto
(r\lambda,sX,s\Theta)
\)
merely reindexes the nonnegative summands in
\eqref{eq:jet-majorant}, and preserves the smallness condition \eqref{eq:theta-small}.
By this covariance, we may assume that both $\lambda$ and $X$ are
dominant. Since $\lambda$ may be assumed to be regular, we have
\begin{equation}\label{eq:positive-simple-frequencies}
\lambda_i:=\langle\lambda,\alpha_i^\vee\rangle>0.
\end{equation}
Since $\lambda$ lies in the root span, 
\begin{equation}\label{eq:norm-simple-frequencies}
\max_i\lambda_i\asymp_\Phi\|\lambda\|.
\end{equation}

\subsection{Spectral and spatial parabolic subgroups}

The double-coset formula \eqref{eq:exact-block}  expresses each block in terms of a
lower-rank alternant, thereby providing the lower-rank input for the
induction. In this subsection, we construct the standard parabolic subgroups
\(P\) and \(Q\) of \(W\) to which that formula will be applied.

The spectral parabolic subgroup \(Q\) will be generated by the simple reflections associated with the relatively singular directions of \(\lambda\), identified by separating the small simple-coroot spectral coordinates \(\langle\lambda,\alpha_i^\vee\rangle\) from those comparable to \(\|\lambda\|\). From now on, we put 
\[M=\norm{\lambda}.\]

\begin{lemma}[Separation of spectral coordinates]\label{lem:spectral}
Given $\eps>0$, there are constants $C_{\Phi,\eps},c_{\Phi,\eps}$ with the following
property. For every regular dominant $\lambda$, there is a proper subset
$J\subsetneq\Delta$ such that, with
\( Q=W_J\),
one has
\begin{align}
 \abs{\ip{\lambda}{\beta^\vee}}
   &\le\eps M,
      &&\beta\in\Phi_Q,\label{eq:lower-frequency}\\
 c_{\Phi,\eps}M
   \le\abs{\ip{q\lambda}{\beta^\vee}}
   &\le C_{\Phi,\eps}M,
      &&q\in Q,\quad\beta\in\Phi\setminus\Phi_Q.
\label{eq:top-frequency}
\end{align}
In particular, if
$|\langle\lambda,\beta^\vee\rangle|>\eps M$
for all $\beta\in\Phi$, we take $J=\varnothing$.
\end{lemma}

\begin{proof}
Order the numbers in \eqref{eq:positive-simple-frequencies} so that
\[
\lambda_{i_1}\leq\cdots\leq \lambda_{i_r}.
\]
Choose \(A=A_{\Phi,\eps}\) sufficiently large so that
$C_\Phi A^{-1}\leq\eps$, 
where \(C_\Phi\) absorbs the constants arising from the simple-coroot
expansions and \eqref{eq:norm-simple-frequencies}. If there is an index
\(\ell<r\) such that
\[
 \lambda_{i_{\ell+1}}>A  \lambda_{i_\ell},
\]
choose the largest such \(\ell\) and put
\[
J=\{\alpha_{i_1},\ldots,\alpha_{i_\ell}\}.
\]
Otherwise, set \(\ell=0\) and \(J=\varnothing\).

By the choice of \(\ell\), 
\[
 \lambda_{i_{j+1}}\leq A  \lambda_{i_j},
\qquad \ell+1\leq j<r.
\]
Consequently,
\[
A^{-(r-j)} \lambda_{i_r}\leq  \lambda_{i_j}\leq  \lambda_{i_r},
\qquad \ell+1\leq j\leq r.
\]
If \(\beta\in\Phi^+\setminus\Phi_Q^+\), then the simple-coroot
expansion of \(\beta^\vee\) contains some
\(\alpha_{i_j}^\vee\) with \(j\geq\ell+1\). The preceding inequalities and the upper and lower
bounds for the nonzero expansion coefficients therefore give
\[
\langle\lambda,\beta^\vee\rangle
\asymp_{\Phi,A} \lambda_{i_r}
\asymp_{\Phi} M.
\]
The group $Q$ preserves $\Phi_Q$ and its complement.  Thus, for
$\beta\notin\Phi_Q$, the root $q^{-1}\beta$ also lies outside $\Phi_Q$; applying the preceding comparison to $q^{-1}\beta$ 
gives
\eqref{eq:top-frequency}.

If \(\ell\geq1\) and \(\beta\in\Phi_Q^+\), then the simple-coroot
expansion of $\beta^\vee$ involves only \(\alpha_{i_1}^\vee,\ldots,
\alpha_{i_\ell}^\vee\). Hence
\[
\langle\lambda,\beta^\vee\rangle
\leq C_\Phi  \lambda_{i_\ell}
< C_\Phi A^{-1} \lambda_{i_{\ell+1}}
\leq C_\Phi A^{-1}M
\leq\eps M.
\]
When \(\ell=0\), the assertion is vacuous because
\(\Phi_Q=\varnothing\). Negative coroots are handled by taking absolute
values. This proves \eqref{eq:lower-frequency}.

\end{proof}

Put
\begin{equation}\label{eq:m-lower}
 m=\norm{\lambda_Q}.
\end{equation}
By \eqref{eq:spectral-split}, for $\beta\in\Phi_Q$ we have
$\langle\lambda_Q,\beta^\vee\rangle
=\langle\lambda,\beta^\vee\rangle$. 
If \(J=\varnothing\), then \(m=0\). Otherwise, applying
\eqref{eq:norm-simple-frequencies} to the root subsystem \(\Phi_Q\) and
using \eqref{eq:lower-frequency}, we obtain
\begin{equation}\label{eq:m-over-M}
m\leq C_\Phi\eps M.
\end{equation}
We choose \(\eps\) later, after fixing the constants in the relevant lower-rank estimates supplied by the induction hypothesis.

Now we construct the spatial parabolic subgroup $P$. 
As $X$ is dominant, $\alpha_i(X)\ge0$.
Motivated by the uncertainty principle, define
\begin{equation}\label{eq:spatial-I}
 I=\{\alpha_i:\alpha_i(X)\le M^{-1}\},
 \qquad P=W_I.
\end{equation}

Related decompositions adapted to root hyperplanes appear
in Marshall's study of higher-rank eigenfunctions and
spherical functions \cite[\S~2.4]{Marshall2016}, as well
as in the second author's earlier work on eigenfunctions
and Schr\"odinger equations on compact Lie groups
\cite{ZhangStrichartz,ZhangFourierRestriction,
ZhangCharacterRestriction}.

\begin{lemma}[Separation of spatial coordinates]\label{lem:spatial}
For suitable $C_\Phi,c_\Phi>0$,
\begin{align}
 \abs{\alpha(X)}&\le C_\Phi M^{-1},
       &&\alpha\in\Phi_P,\label{eq:P-small}\\
 \abs{\alpha(X)}&\ge c_\Phi M^{-1},
       &&\alpha\in\Phi\setminus\Phi_P.\label{eq:P-large}
\end{align}
\end{lemma}

\begin{proof}
This follows immediately by writing each positive root as a nonnegative
linear combination of simple roots and using the definition of \(P\);
negative roots are handled by taking absolute values.
\end{proof}

\subsection{Setup for the induction}

We prove Proposition~\ref{prop:local-jets} by induction on $\rank\Phi$, with all
derivative orders available at each lower rank.  
For the formal rank-zero case, $\Phi=\Theta=\varnothing$ and
$A_0^\Phi=\pi_\Theta=1$, so the statement is immediate.
By our discussion in Section \ref{subsec: reduction}, we may assume that $\lambda$ is regular, and that both $\lambda$ and $X$ are
dominant.

For a fixed target order \(k\), only lower-rank derivative orders at most
\(k+|\Phi^+|\) will occur, so only finitely many constants from the
relevant lower-rank estimates are involved. Apply
Lemma~\ref{lem:spectral} with a parameter \(\eps>0\), to be chosen in the next subsection
in terms of these lower-rank constants, and let \(Q=W_J\) be the resulting spectral
parabolic subgroup. 
Since $J\subsetneq\Delta$, $\Phi_Q$ is a proper parabolic root subsystem
with $\rank\Phi_Q=|J|<\rank\Phi$, providing the required decrease
in rank for the induction.
Then choose the spatial parabolic subgroup $P=W_I$ by
\eqref{eq:spatial-I}.  Taking the constant in
\eqref{eq:theta-small} small
ensures
\begin{equation}\label{eq:theta-in-P}
 \Theta\subset\Phi_P.
\end{equation}
Indeed, for $\alpha\in\Theta^+=\Theta\cap\Phi^+$, dominance of $X$
and the nonnegative coefficients in the simple-root expansion of
$\alpha$ show that every simple root $\gamma$ in the support of
$\alpha$ satisfies
\[
  \gamma(X)\leq C_\Phi c_{\Phi,k}M^{-1}.
\]  
With
$c_{\Phi,k}$ small enough, the support of \(\alpha\) is
therefore contained in \(I\), so
$\alpha\in\Phi_P$.  This proves
\eqref{eq:theta-in-P}.

\subsection{Derivative estimate for a double-coset block}
\label{subsec:differentiated-block}

The goal of this subsection is to establish, under the induction
hypothesis, the derivative estimate for the quotient
$B_{u,\lambda}/\pi_\Theta$. We combine the exact block identity
\eqref{eq:exact-block} with the iterated integral representation
\eqref{eq:integral} for divided-difference operators and then apply the
induction hypothesis to the lower-rank alternant quotient.

Applying
Lemma~\ref{lem:evaluation} with the ambient root system \(\Phi\) and
\(W_I=P\), we obtain, for every \(p\in P\) and every evaluation point
\(Y=Y(\mathbf t,X)\) occurring in the integral representation of
\(\partial_p f(X)\),
\begin{equation}\label{eq:evaluation-all-p}
 \|Y-X\|
 \leq C_\Phi\max_{\alpha\in I}|\alpha(X)|
 \leq C_\Phi M^{-1}.
\end{equation}

The following lemma records the root estimates needed to apply
the induction hypothesis at these evaluation points and subsequently
to match the resulting bound with the desired block majorant.  
Recall the intersection subgroups $R,R'$ from Section \ref{sec: double cosets}. 

\begin{lemma}[Root estimates at the evaluation points]
\label{lem:block-root-estimates}
Let \(Y=Y(\mathbf t,X)\) be an evaluation point
as above.  If \(\beta\in\Phi_Q\setminus\Phi_{R'}\), then
\begin{align}
 |(u\beta)(X)|
   &\geq c_\Phi M^{-1},\label{eq:endpoint-lower}\\
 |(u\beta)(Y)|
   &\leq C_\Phi |(u\beta)(X)|,\label{eq:path-upper}
\end{align}
and, for every \(q\in Q\),
\begin{equation}\label{eq:absorption}
 \frac{|\langle q\lambda_Q,\beta^\vee\rangle|}{M}
 \leq C_\Phi
 \min\{1,|(u\beta)(X)||\langle q\lambda_Q,\beta^\vee\rangle|\}.
\end{equation}
If \(\beta\in\Phi_{R'}\), then
\begin{equation}\label{eq:lower-rank-smallness}
 |(u\beta)(Y)|\leq C_\Phi M^{-1}.
\end{equation}
\end{lemma}

\begin{proof}
Let $\beta\in\Phi_Q$. By \eqref{eq:evaluation-all-p}, we have 
\begin{align}\label{eq: ubetaY-X}
 |(u\beta)(Y)-(u\beta)(X)|
 \leq C_\Phi\|Y-X\|
 \leq C_\Phi M^{-1}. 
\end{align}
The intersection identities in
Lemma~\ref{lem:double-cosets} give
\begin{equation}\label{eq:intersection-root-equivalence}
 u\beta\in\Phi_P
 \quad\Longleftrightarrow\quad
 u\beta\in\Phi_P\cap u\Phi_Q=\Phi_R
 \quad\Longleftrightarrow\quad
 \beta\in\Phi_{R'}.
\end{equation}
Therefore, if \(\beta\notin\Phi_{R'}\), then \(u\beta\notin\Phi_P\), and
\eqref{eq:endpoint-lower} follows from Lemma~\ref{lem:spatial}. Then \eqref{eq:path-upper} follows from \eqref{eq:endpoint-lower}  and \eqref{eq: ubetaY-X}.

Since \(q\) acts orthogonally,
\( |\langle q\lambda_Q,\beta^\vee\rangle|\leq C_\Phi \|\lambda_Q\|\leq C_\Phi M\). 
Together with \eqref{eq:endpoint-lower}, this implies
\eqref{eq:absorption}.

Finally, if \(\beta\in\Phi_{R'}\), then
\(u\beta\in\Phi_P\).  Lemma~\ref{lem:spatial} therefore gives
$|(u\beta)(X)|\leq C_\Phi M^{-1}$. Together with \eqref{eq: ubetaY-X}, this implies
\eqref{eq:lower-rank-smallness}.
\end{proof}

Since the positive systems are induced from \(\Phi^+\), by \eqref{eq:theta-in-P}, 
\[
 \Theta^+=\Theta\cap\Phi^+
 \subset\Phi_P\cap\Phi^+
 =\Phi_P^+.
\]
For
\[
 S_P\subset\Phi_P^+\setminus\Theta^+,
 \qquad
 S_Q\subset\Phi_Q^+\setminus\Phi_{R'}^+,
\]
define
\begin{align}
 \mathcal B_{u,q}^{S_P,S_Q}(X)
 :={}&M^n
 \prod_{\alpha\in
   \Phi_P^+\setminus(\Theta^+\cup S_P)}
   |\alpha(X)|
 \prod_{\beta\in\Phi_{R'}^+\cup S_Q}
   |\langle q\lambda_Q,\beta^\vee\rangle|
 \notag\\[-1mm]
 &\times
 \prod_{\beta\in
   \Phi_Q^+\setminus(\Phi_{R'}^+\cup S_Q)}
 \min\{1,|(u\beta)(X)||\langle q\lambda_Q,\beta^\vee\rangle|\}.
\label{eq:block-majorant}
\end{align}

\begin{proposition}
\label{prop:differentiated-block} Assume that Proposition~\ref{prop:local-jets} holds for every root system of rank
strictly less than \(\rank\Phi\), for all derivative orders at most
\(k+|\Phi^+|\). Then for all unit vectors \(\xi_1,\ldots,\xi_k\in V\), one has
\begin{align}
 \left|
 D_{\xi_1}\cdots D_{\xi_k}
 \left(\frac{B_{u,\lambda}}{\pi_\Theta}\right)(X)
 \right|
 &\leq C_{\Phi,k}
 \sum_{q\in Q}
 \sum_{\substack{
   S_P\subset\Phi_P^+\setminus\Theta^+\\
   S_Q\subset\Phi_Q^+\setminus\Phi_{R'}^+\\
   |S_P|+|S_Q|\leq k}}
 M^{k-|S_P|-|S_Q|}
 \mathcal B_{u,q}^{S_P,S_Q}(X).
\label{eq:differentiated-block}
\end{align}
\end{proposition}

Later, we construct the deleted set $S=S_P\sqcup uS_Q$,
so the natural restriction $|S_P|+|S_Q|\leq k$
corresponds to $|S|\leq k$.

\begin{proof}
Division of
\eqref{eq:exact-block} by \(\pi_\Theta\) gives
\begin{equation}\label{eq:normalized-block}
 \frac{B_{u,\lambda}(X)}{\pi_\Theta(X)}
 =
 \det(u)
 \frac{\pi_P(X)}{\pi_\Theta(X)}
 \partial_{w_P^R}G_u(X).
\end{equation}
Consider one term in the Leibniz expansion in which \(a\)
derivatives fall on \(\pi_P/\pi_\Theta\), and put \(b=k-a\).
Because
\[
 \frac{\pi_P(X)}{\pi_\Theta(X)}
 =
 \prod_{\alpha\in\Phi_P^+\setminus\Theta^+}\alpha(X)
\]
is a product of distinct linear forms, a nonzero such term has the
form
\begin{equation}\label{eq:external-prefactor}
 c_{S_P,\boldsymbol\xi}
 \prod_{\alpha\in
   \Phi_P^+\setminus(\Theta^+\cup S_P)}
 \alpha(X),
 \qquad
 S_P\subset\Phi_P^+\setminus\Theta^+,
 \qquad |S_P|=a,
\end{equation}
where \(\boldsymbol{\xi}=(\xi_1,\ldots,\xi_k)\), and
\(|c_{S_P,\boldsymbol{\xi}}|\leq C_{\Phi,k}\),
since each \(\xi_j\) is a unit vector.

Recall \eqref{eq:wpR}. Choose a reduced expression of \(w_P^R\) of length
\(
 n=\ell(w_P^R)=|\Phi_P^+|-|\Phi_R^+|\). 
By the iterated integral representation \eqref{eq:integral}, applied with
\(w=w_P^R\) and \(f=G_u\),
\[
 \partial_{w_P^R}G_u(X)
 =
 \int_{[0,1]^n}
 D_{v_1(\mathbf t)}\cdots D_{v_n(\mathbf t)}
 G_u\bigl(Y(\mathbf t,X)\bigr)\,\dd\mathbf t.
\]
Here the directions \(v_j(\mathbf t)\) and the linear map
\(X\mapsto Y(\mathbf t,X)\) are those defined in Lemma \ref{lem:integral} and have norms bounded in terms of \(\Phi\).
Passing the remaining \(b\) derivatives under the integral and using the chain rule \( D_\xi\{F(TX)\}=D_{T\xi}F(TX)\), the resulting integrand has the form 
\begin{equation}\label{eq:external-ordinary}
 D_{\zeta_1}\cdots D_{\zeta_{n+b}}G_u(Y),
 \qquad
 \|\zeta_r\|\leq C_\Phi.
\end{equation}

Write
\[
 G_u(Y)=e_u(Y)H_u(Y),
 \qquad
 e_u(Y)=e^{i\langle \lambda_Q^\perp,u^{-1}Y\rangle},
 \qquad
 H_u(Y)=
 h_Q(u^{-1}Y),
\]
where
\[
 h_Q(Z):=\frac{A_{\lambda_Q}^{\Phi_Q}(Z)}{\pi_{R'}(Z)}.
\]
 Consequently, for any directions
$\zeta_1,\ldots,\zeta_j$,
\begin{align}
 D_{\zeta_1}\cdots D_{\zeta_j}H_u(Y)
 =
 D_{u^{-1}\zeta_1}\cdots D_{u^{-1}\zeta_j}
 h_Q(u^{-1}Y).
\label{eq:lower-rank-chain-rule}
\end{align}
Suppose that \(j\) of the \(n+b\) derivatives in
\eqref{eq:external-ordinary} fall on \(H_u\).  The remaining
\(n+b-j\) derivatives fall on \(e_u\). Since
\(\|\lambda_Q^\perp\|\leq\|\lambda\|=M\), each derivative falling on
\(e_u\) contributes at most \(C_\Phi M\), and their total contribution is
bounded by
\[
C_{\Phi,k}M^{n+b-j}.
\]
Since \(u\) is orthogonal, the directions $u^{-1}\zeta_1,\ldots,u^{-1}\zeta_j$  have the same norms as the
directions $\zeta_1,\ldots,\zeta_j$, which are uniformly bounded by \eqref{eq:external-ordinary}.  The nonzero directions are normalized before
applying the induction hypothesis, with their norms absorbed into the
constant.

For $\beta\in\Phi_{R'}$, equations
\eqref{eq:lower-rank-smallness} and \eqref{eq:m-over-M} give
\[
m|\beta(u^{-1}Y)|
=m|(u\beta)(Y)|
\leq C_\Phi\frac{m}{M}
\leq C_\Phi\varepsilon.
\]
Choose $0<\varepsilon=\varepsilon_{\Phi,k}\leq1$ sufficiently
small that $C_\Phi\varepsilon$ is below the smallness constants
in \eqref{eq:theta-small} for every proper parabolic root
subsystem of $\Phi$ and every derivative order at most
$k+|\Phi^+|$. There are only finitely many such constants.
The induction hypothesis therefore applies to $h_Q$ at every point
$u^{-1}Y$, for all orders $0\leq j\leq n+b$.

Since \(n+b\leq |\Phi^+|+k\), the number of terms in the Leibniz
expansion is bounded in terms of \(\Phi\) and \(k\). Combining the
preceding discussion with the lower-rank derivative estimate for \(h_Q\)
at each order \(0\leq j\leq n+b\), and summing over these Leibniz terms,
we obtain
\begin{align}
\bigl|D_{\zeta_1}\cdots D_{\zeta_{n+b}}G_u(Y)\bigr|
&\leq
C_{\Phi,k}
\sum_{j=0}^{n+b}
\sum_{q\in Q}
\sum_{\substack{
 S_{\mathrm{low}}
 \subset\Phi_Q^+\setminus\Phi_{R'}^+\\
 |S_{\mathrm{low}}|\leq j}}
 M^{n+b-j}m^{j-|S_{\mathrm{low}}|}
 \notag\\[-1mm]
&\quad\times
\prod_{\beta\in\Phi_{R'}^+\cup S_{\mathrm{low}}}
 |\langle q\lambda_Q,\beta^\vee\rangle|
 \notag\\[-1mm]
&\quad\times
\prod_{\beta\in
 \Phi_Q^+\setminus
 (\Phi_{R'}^+\cup S_{\mathrm{low}})}
 \min\{1,|(u\beta)(Y)||\langle q\lambda_Q,\beta^\vee\rangle|\}.
\label{eq:external-lower-term}
\end{align}

For each \(S_{\mathrm{low}}\), put \(s=|S_{\mathrm{low}}|\).
Not all roots in $S_{\mathrm{low}}$ need contribute to
the final deleted set $S$.
We choose $S_Q\subseteq S_{\mathrm{low}}$
to specify which deletions are retained. The lower-rank estimate allows
\(s\leq j\leq n+b\), so \(s\) can exceed \(b\). Since $|S_P|=a$, the restriction $|S_P|+|S_Q|\leq k$
permits at most $b=k-a$ of these deletions to remain. 
We therefore retain all \(s\) deletions when \(s\leq b\) and
retain \(b\) of them when \(s>b\). Accordingly, set
\(t=\min\{s,b\}\), choose any subset
\[
 S_Q\subseteq S_{\mathrm{low}},
 \qquad |S_Q|=t,
\]
and define the set of ``absorbed'' roots by
\[
 S_{\mathrm{abs}}:=S_{\mathrm{low}}\setminus S_Q.
\]
Since \(|S_{\mathrm{abs}}|=s-t\), we have the exact identity
\begin{align}
 &M^{n+b-j}m^{j-s}
   \prod_{\beta\in S_{\mathrm{low}}}
   |\langle q\lambda_Q,\beta^\vee\rangle|
 \notag\\
 &\qquad=
 M^nM^{b-t}
 \prod_{\beta\in S_Q}|\langle q\lambda_Q,\beta^\vee\rangle|
 \left(\frac{m}{M}\right)^{j-s}
 \prod_{\beta\in S_{\mathrm{abs}}}
 \frac{|\langle q\lambda_Q,\beta^\vee\rangle|}{M}.
 \label{eq:bookkeeping-ratio}
\end{align}
Since \(s\leq j\), our choice of \(\varepsilon\) and
\eqref{eq:m-over-M} give
\[
 \left(\frac{m}{M}\right)^{j-s}
 \leq (C_\Phi\varepsilon)^{j-s}
 \leq C_{\Phi,k}.
\]

We call the roots in \(S_{\mathrm{abs}}\) ``absorbed'' because
\eqref{eq:absorption} bounds each ratio
\(|\langle q\lambda_Q,\beta^\vee\rangle|/M\) in the last product
of \eqref{eq:bookkeeping-ratio} by a constant times the
corresponding minimum factor at \(X\). This restores the minimum
factor deleted for \(\beta\) in the lower-rank estimate. Since
\eqref{eq:absorption} applies to every root in \(S_{\mathrm{low}}\),
the choice of the \(t\) roots retained in \(S_Q\) is arbitrary.
For each undeleted root
\(\beta\in\Phi_Q^+\setminus(\Phi_{R'}^+\cup S_{\mathrm{low}})\),
we use \eqref{eq:path-upper} to bound its minimum factor at \(Y\)
by a constant times the corresponding factor at \(X\).

After multiplication by the prefactor in
\eqref{eq:external-prefactor}, each summand in
\eqref{eq:external-lower-term} is bounded, for the corresponding
sets \(S_P\) and \(S_Q\), by
\[
 C_{\Phi,k}M^{b-t}\mathcal B_{u,q}^{S_P,S_Q}(X).
\]
Since \(b=k-a\), \(a=|S_P|\), and \(t=|S_Q|\), the exponent
of \(M\) is
\[
 b-t=k-|S_P|-|S_Q|\geq0,
\]
as required in \eqref{eq:differentiated-block}. Since there are at most \(2^{|\Phi_Q^+|}\) choices for \(S_{\mathrm{low}}\)
and at most \(k+|\Phi^+|\) derivatives in each Leibniz expansion,
each pair \((S_P,S_Q)\) occurs with multiplicity bounded in terms of $\Phi$ and $k$. Absorbing the multiplicities into \(C_{\Phi,k}\) and summing gives
\eqref{eq:differentiated-block}. 
\end{proof}

\subsection{Matching the ambient root factors}
\label{subsec:root-factor-matching}

For \(q\in Q\), put \(w=uq\). For $\alpha\in\Phi^+$, put $\beta=u^{-1}\alpha$. 
The four root classes in
\eqref{eq:four-root-classes} have the following spatial and spectral
behavior. 

\begin{lemma}[Spatial and spectral sizes for the four root classes]

\label{lem:four-class-estimates}

\begingroup
\renewcommand{\arraystretch}{1.35}
\setlength{\arraycolsep}{9pt}
\[
\begin{array}{c|c|c}
 \text{root class}
 & \text{spatial size}
 & \text{spectral size}
 \\ \hline
 \alpha\in \Phi_P^+\setminus\Phi_R^+
 & |\alpha(X)|\lesssim_\Phi M^{-1}
 & |\langle w\lambda,\alpha^\vee\rangle|\asymp_{\Phi,k} M
 \\
 \alpha\in \Phi_R^+
 & |\alpha(X)|\lesssim_\Phi M^{-1}
 & |\langle w\lambda,\alpha^\vee\rangle| =
 |\langle q\lambda_Q,\beta^\vee\rangle|
   \lesssim_\Phi m
 \\
 \alpha\in u(\Phi_Q^+\setminus\Phi_{R'}^+)
 & |\alpha(X)|\gtrsim_\Phi M^{-1}
 & |\langle w\lambda,\alpha^\vee\rangle| =
 |\langle q\lambda_Q,\beta^\vee\rangle|
  \lesssim_\Phi m
 \\
 \alpha\in \Phi^+\setminus(\Phi_P^+\cup u\Phi_Q^+)
 & |\alpha(X)|\gtrsim_\Phi M^{-1}
 & |\langle w\lambda,\alpha^\vee\rangle|\asymp_{\Phi,k} M
\end{array}
\]
\endgroup
\end{lemma}

\begin{proof}
Recall Figure~\ref{fig:four-root-classes} for the four root classes. 
The spatial bounds follow from Lemma~\ref{lem:spatial}. 

For the spectral bounds, we apply Lemma~\ref{lem:spectral}.
If \(\alpha=u\beta\in u\Phi_Q^+\), then Weyl invariance and the
decomposition \(\lambda=\lambda_Q^\perp+\lambda_Q\) give
\[
 |\langle w\lambda,\alpha^\vee\rangle|
  =
 |\langle q\lambda_Q,\beta^\vee\rangle|
 \lesssim_\Phi m.
\]
This proves the spectral assertions for the second and third classes.
On the other hand, assume \(\alpha\in\Phi^+\setminus u\Phi_Q^+\). 
By \eqref{eq:u-positivity}, \(u\Phi_Q^+\subset\Phi^+\). Thus \(\alpha\notin u\Phi_Q\).  Applying \eqref{eq:top-frequency} to
\(\beta=u^{-1}\alpha\) gives
\(
 |\langle w\lambda,\alpha^\vee\rangle|
 =
 |\langle q\lambda,\beta^\vee\rangle|
 \asymp_{\Phi,k} M\). 
This proves the spectral assertions for the first and fourth classes.
\end{proof}

Let \(S_P,S_Q\) be as in Proposition~\ref{prop:differentiated-block}. Recall 
\begin{align}\label{eq:SP-SQ-recall}
S_P\subset\Phi_P^+\setminus\Theta^+,\qquad  uS_Q\subset u(\Phi_Q^+\setminus\Phi_{R'}^+)\subset \Phi^+\setminus\Phi^+_P.
\end{align}
Then we may define the disjoint union
\begin{equation}\label{eq:final-deleted-set}
 S:=S_P\sqcup uS_Q\subset \Phi^+\setminus\Theta^+.
\end{equation}

\begin{lemma}[Root-factor matching]
\label{lem:deleted-factor-matching}
\begin{align}
 \mathcal B_{u,q}^{S_P,S_Q}(X)
 \leq C_{\Phi,k}
 \prod_{\alpha\in\Theta^+\cup S}
   |\langle w\lambda,\alpha^\vee\rangle|
 \prod_{\alpha\in
   \Phi^+\setminus(\Theta^+\cup S)}
   \min\{1,\abs{\alpha(X)}\abs{\ip{w\lambda}{\alpha^\vee}}\}.
\label{eq:deleted-factor-matching}
\end{align}
\end{lemma}

\begin{proof}
Set
\[
\begin{aligned}
 \mathcal C_1&=\Phi_P^+\setminus\Phi_R^+,
 &\qquad
 \mathcal C_2&=\Phi_R^+,\\
 \mathcal C_3&=u(\Phi_Q^+\setminus\Phi_{R'}^+),
 &
 \mathcal C_4&=\Phi^+\setminus(\Phi_P^+\cup u\Phi_Q^+),
\end{aligned}
\]
and put
\[
 D=\Theta^+\sqcup S.
\]
By \eqref{eq:SP-SQ-recall} and \eqref{eq:theta-in-P},
\[
 D\cap\Phi_P^+=\Theta^+\cup S_P,
 \qquad
 D\cap\mathcal C_3=uS_Q,
 \qquad
 D\cap\mathcal C_4=\varnothing.
\]

We estimate separately the contribution from each of the four root
classes. Since
$n=|\mathcal C_1|$, 
the factor \(M^n\) supplies one copy of \(M\) for every root in
\(\mathcal C_1\).  
For \(\alpha\in\mathcal C_1\), the preceding lemma gives
\[
 |\langle w\lambda,\alpha^\vee\rangle|\asymp_{\Phi,k} M
 \qquad\text{and}\qquad
 |\alpha(X)|\,|\langle w\lambda,\alpha^\vee\rangle|
 \lesssim_{\Phi,k} 1.
\]
Hence 
\[
 M|\alpha(X)|
 \asymp_{\Phi,k}
 |\alpha(X)|\,|\langle w\lambda,\alpha^\vee\rangle|
 \lesssim_{\Phi,k} \min\{1,\abs{\alpha(X)}\abs{\ip{w\lambda}{\alpha^\vee}}\}.
\]
Using the bound $ |\langle w\lambda,\alpha^\vee\rangle|\asymp_{\Phi,k} M$ for \(\alpha\in\mathcal C_1\cap D\),
and the
above estimate for \(\alpha\in\mathcal C_1\setminus D\), 
noting that $\mathcal C_1\setminus D= \Phi_P^+\setminus(\Phi_R^+\cup\Theta^+\cup S_P)$, we have
\begin{align}\label{eq:matching-P-minus-R}
& M^n
 \prod_{\alpha\in \Phi_P^+\setminus(\Phi_R^+\cup\Theta^+\cup S_P)}|\alpha(X)|\nonumber\\
& \lesssim_{\Phi,k}
 \prod_{\alpha\in\mathcal C_1\cap D}
   |\langle w\lambda,\alpha^\vee\rangle|
 \prod_{\alpha\in\mathcal C_1\setminus D}
   \min\{1,\abs{\alpha(X)}\abs{\ip{w\lambda}{\alpha^\vee}}\}.
\end{align}

For \(\alpha\in\mathcal C_2\), 
\(\beta=u^{-1}\alpha\in\Phi_{R'}^+\). The preceding lemma and \eqref{eq:m-over-M} give
\[
 |\alpha(X)|\,|\langle q\lambda_Q,\beta^\vee\rangle|
 \lesssim_\Phi \frac{m}{M}
 \lesssim_\Phi\varepsilon
 \leq 1.
\]
It follows that
\( |\alpha(X)|\,|\langle q\lambda_Q,\beta^\vee\rangle|
 \lesssim_\Phi \min\{1,\abs{\alpha(X)}\abs{\ip{w\lambda}{\alpha^\vee}}\}\). 
Consequently, noting that 
$\mathcal C_2\setminus D=\Phi_R^+\setminus(\Theta^+\cup S_P)$, we have 
\begin{align}\label{eq:matching-R}
 &\prod_{\alpha\in\Phi_R^+\setminus(\Theta^+\cup S_P)}|\alpha(X)|
 \prod_{\beta\in\Phi_{R'}^+}
   |\langle q\lambda_Q,\beta^\vee\rangle|\nonumber\\
 &\lesssim_\Phi
 \prod_{\alpha\in\mathcal C_2\cap D}
   |\langle w\lambda,\alpha^\vee\rangle|
 \prod_{\alpha\in\mathcal C_2\setminus D}
   \min\{1,\abs{\alpha(X)}\abs{\ip{w\lambda}{\alpha^\vee}}\}.
\end{align}

For the third class, 
\(\beta=u^{-1}\alpha\in\Phi_Q^+\setminus\Phi_{R'}^+\). 
 Note that 
 \[ \min\{1,|(u\beta)(X)||\langle q\lambda_Q,\beta^\vee\rangle|\}
 =
 \min\{1,\abs{\alpha(X)}\abs{\ip{w\lambda}{\alpha^\vee}}\},\]
 which we use for $\beta\notin S_Q$. On the other hand, \(\beta\in S_Q\) is equivalent to \(u\beta\in D\).
It follows that
\begin{equation}\label{eq:matching-Q-minus-Rprime}
\begin{aligned}
 &\prod_{\beta\in S_Q}|\langle q\lambda_Q,\beta^\vee\rangle|
 \prod_{\beta\in
   \Phi_Q^+\setminus(\Phi_{R'}^+\cup S_Q)}
   \min\{1,|(u\beta)(X)||\langle q\lambda_Q,\beta^\vee\rangle|\}
 \\
 &\qquad =
 \prod_{\alpha\in\mathcal C_3\cap D}
   |\langle w\lambda,\alpha^\vee\rangle|
 \prod_{\alpha\in\mathcal C_3\setminus D}
   \min\{1,\abs{\alpha(X)}\abs{\ip{w\lambda}{\alpha^\vee}}\}.
\end{aligned}
\end{equation}

Finally, for \(\alpha\in\mathcal C_4\), the preceding lemma gives 
\( |\alpha(X)| |\langle w\lambda,\alpha^\vee\rangle|\gtrsim_{\Phi,k} 1\),
so that
\( \min\{1,\abs{\alpha(X)}\abs{\ip{w\lambda}{\alpha^\vee}}\}\asymp_{\Phi,k} 1\). 
We obtain
\begin{equation}\label{eq:matching-exterior}
 1\lesssim_{\Phi,k}
 \prod_{\alpha\in\mathcal C_4}\min\{1,\abs{\alpha(X)}\abs{\ip{w\lambda}{\alpha^\vee}}\}.
\end{equation}

Multiplying the four classwise comparisons
\eqref{eq:matching-P-minus-R}--\eqref{eq:matching-exterior} and using
the disjoint decomposition \eqref{eq:four-root-classes}, we obtain the desired estimate \eqref{eq:deleted-factor-matching}. 
\end{proof}

\subsection{Completion of the induction}
\label{subsec:completion-induction}

Since
\[
 A_\lambda^\Phi=\sum_{u\in\mathcal D}B_{u,\lambda},
\]
Proposition~\ref{prop:differentiated-block} and
Lemma~\ref{lem:deleted-factor-matching} give
\begin{align*}
 \left|
   D_{\xi_1}\cdots D_{\xi_k}
   \left(\frac{A_\lambda^\Phi}{\pi_\Theta}\right)(X)
  \right| 
 &\leq C_{\Phi,k}
   \sum_{u\in\mathcal D}\sum_{q\in Q}
   \sum_{\substack{
     S_P\subset\Phi_P^+\setminus\Theta^+\\
     S_Q\subset\Phi_Q^+\setminus\Phi_{R'}^+\\
     |S_P|+|S_Q|\leq k}}
   M^{k-|S|}
   \prod_{\alpha\in\Theta^+\cup S}
     |\langle uq\lambda,\alpha^\vee\rangle| \\
 &\quad\times
   \prod_{\alpha\in\Phi^+\setminus(\Theta^+\cup S)}
     \min\{1,\abs{\alpha(X)}
     \abs{\ip{uq\lambda}{\alpha^\vee}}\},
\end{align*}
where \(S=S_P\sqcup uS_Q\). 
For fixed \(u\), the pair \((S_P,S_Q)\) is determined by \(S\), because
\(S_P=S\cap\Phi_P^+\) and \(uS_Q=S\setminus\Phi_P^+\). The sum over these pairs may therefore be enlarged to a sum over all
subsets \(S\subset\Phi^+\setminus\Theta^+\) satisfying \(|S|\leq k\). Moreover, the Weyl elements appearing
above are only \(uq\), corresponding to \(p=e\) in each double coset
\(PuQ\).  Since
\[
 \bigsqcup_{u\in\mathcal D}uQ\subset W
\]
and all summands are nonnegative, enlarging the sum once more gives
\[
 \left|
 D_{\xi_1}\cdots D_{\xi_k}
 \left(\frac{A_\lambda^\Phi}{\pi_\Theta}\right)(X)
 \right|
 \leq
 C_{\Phi,k}\mathcal M_{\Phi,\Theta}^{(k)}(\lambda,X).
\]
This completes the simultaneous
induction on the rank and proves Proposition~\ref{prop:local-jets} when the spectral parameters lie in the root span.

\begin{remark}
Although a general element of $PuQ$ has the form $puq$, the preceding
matching uses only $w=uq$, corresponding to $p=e$.
An analogous matching holds for any $w=puq$ with $p\in P$.

Indeed, $p$ preserves $\Phi^+\setminus\Phi_P^+$, and
the displacement estimate
$\|p^{-1}X-X\|\lesssim_\Phi M^{-1}$
together with Lemma~\ref{lem:spatial} gives
\[
|(p\alpha)(X)|\asymp_\Phi|\alpha(X)|,
\qquad \alpha\in\Phi^+\setminus\Phi_P^+.
\]
Thus replacing $u$ by $pu$ in the expression for
$\mathcal B_{u,q}^{S_P,S_Q}(X)$ gives a comparable quantity. 
The same matching argument applies after reindexing the four
root classes, with $S=S_P\sqcup puS_Q$.

Since the desired majorant is a nonnegative sum over all $w\in W$,
it suffices to use the matching with $p=e$.
\end{remark}

We finally remove the restriction that the spectral parameter lie in
the root span. Write the orthogonal decomposition
$\lambda=\lambda_E+\lambda_\perp$, where
$E=\operatorname{span}_{\mathbb R}(\Phi)$. Since $W$ fixes $\lambda_\perp$,
\[
\frac{A_\lambda^\Phi(X)}{\pi_\Theta(X)}
=e^{i\langle\lambda_\perp,X\rangle}
\frac{A_{\lambda_E}^\Phi(X)}{\pi_\Theta(X)}.
\]
Moreover,
$\langle w\lambda_E,\alpha^\vee\rangle
=\langle w\lambda,\alpha^\vee\rangle$ for all $w\in W$ and
$\alpha\in\Phi$. Shrink the constant in the smallness condition \eqref{eq:theta-small} to the minimum of
the constants already obtained for derivative orders $0\leq j\leq k$.
In a Leibniz term with $j$ derivatives falling on the alternant quotient,
the spectral power is bounded by
\[
\|\lambda_\perp\|^{k-j}
\|\lambda_E\|^{j-|S|}
\leq\|\lambda\|^{k-|S|}.
\]
The smallness condition for $\lambda_E$ follows from that for
$\lambda$, since $\|\lambda_E\|\leq\|\lambda\|$. Applying the estimates already proved and summing the finitely many
Leibniz terms therefore proves Proposition~\ref{prop:local-jets}
for every $\lambda\in V^*$.

\section{Global derivative estimates and bounds for spherical functions}
\label{sec:real-and-spherical-consequences}

\subsection{Proof of the global derivative estimate}
\label{subsec:global-alternant-jets}

\begin{proof}[Proof of Theorem~\ref{thm:jets}]
If $\Phi=\varnothing$, then $\pi=1$ and
$A_\lambda^\Phi(X)=e^{i\langle\lambda,X\rangle}$, so the estimate is
immediate. If $\Phi\neq\varnothing$ and $\lambda$ is singular, then
$A_\lambda^\Phi\equiv0$, and the estimate is again immediate.
We may therefore assume that $\lambda$ is regular and put
$M=\|\lambda\|>0$. By Weyl covariance, with the derivative directions transformed
together with $X$ as in \eqref{eq: covariance}, and by the Weyl
invariance of the majorant, we may assume that $X$ is dominant.

Choose a sufficiently small constant $\delta_{\Phi,k}>0$ such that for
\[
 I=\{\alpha_i\in\Delta:
       \alpha_i(X)\leq\delta_{\Phi,k}M^{-1}\},
 \qquad
 \Theta=\Phi_I,
\]
we have 
\begin{equation}\label{eq:theta-small-for-global-jets}
 M|\alpha(X)|
 \leq\min_{0\leq j\leq k}c_{\Phi,j},
 \qquad \alpha\in\Theta^+.
\end{equation}
This choice of $\Theta$ is held fixed when taking derivatives at $X$. 
Thus Proposition~\ref{prop:local-jets} applies to
$A_\lambda^\Phi/\pi_\Theta$ for every derivative order at most $k$.
On the other hand, the nonnegative integer coefficients in the
simple-root expansions give
\begin{equation}\label{eq:outside-theta-spatial-lower}
 |\alpha(X)|\geq\delta_{\Phi,k}M^{-1},
 \qquad
 \alpha\in\Phi^+\setminus\Theta^+.
\end{equation}

Write
\[
 p_\Theta(X)
 :=\prod_{\alpha\in\Phi^+\setminus\Theta^+}\alpha(X),
 \qquad
 \pi(X)=\pi_\Theta(X)p_\Theta(X).
\]
It follows from \eqref{eq:outside-theta-spatial-lower} that, for unit
directions $\eta_1,\ldots,\eta_r$ and $0\leq r\leq k$,
\begin{equation}\label{eq:reciprocal-large-root-product}
 \left|
 D_{\eta_1}\cdots D_{\eta_r}p_\Theta(X)^{-1}
 \right|
 \leq
 C_{\Phi,k}M^r
 \prod_{\alpha\in\Phi^+\setminus\Theta^+}
 |\alpha(X)|^{-1}.
\end{equation}

Apply the Leibniz rule to the identity
\[
 \frac{A_\lambda^\Phi}{\pi}
 =\frac{A_\lambda^\Phi}{\pi_\Theta}\,p_\Theta^{-1},
\]
which holds on a neighborhood of $X$.
Consider a term in which $j$ derivatives fall on
$A_\lambda^\Phi/\pi_\Theta$. By Proposition~\ref{prop:local-jets} and
\eqref{eq:reciprocal-large-root-product}, the absolute value of
this Leibniz term is bounded by a sum over $w\in W$ and
$S\subset\Phi^+\setminus\Theta^+$ with $|S|\leq j$, and the summand indexed by $(w,S)$ is
\begin{equation}\label{eq:global-jet-intermediate}
\begin{aligned}
 &C_{\Phi,k}
 \left(M^{k-j}
 \prod_{\alpha\in\Phi^+\setminus\Theta^+}|\alpha(X)|^{-1}\right)
 M^{j-|S|}
 \prod_{\alpha\in\Theta^+\cup S}
   |\langle w\lambda,\alpha^\vee\rangle|
 \\
 &\qquad\times
 \prod_{\alpha\in\Phi^+\setminus(\Theta^+\cup S)}
 \min\left\{1,\,
 |\alpha(X)|\,|\langle w\lambda,\alpha^\vee\rangle|\right\}
 \\
 &= C_{\Phi,k}M^{k-|S|}
 \prod_{\alpha\in\Phi^+}
   |\langle w\lambda,\alpha^\vee\rangle|
 \prod_{\alpha\in S}|\alpha(X)|^{-1}
 \\
 &\qquad\times
 \prod_{\alpha\in\Phi^+\setminus(\Theta^+\cup S)}
 \min\left\{1,\,
 \bigl(|\alpha(X)|\,|\langle w\lambda,\alpha^\vee\rangle|\bigr)^{-1}
 \right\}.
\end{aligned}
\end{equation}
Here we used $\min\{1,ab\}/a=b\min\{1,(ab)^{-1}\}$ for $a,b>0$.

If $\alpha\in\Theta^+$, then
\eqref{eq:theta-small-for-global-jets} and $|\langle w\lambda,\alpha^\vee\rangle|\leq C_\Phi M$ give $|\alpha(X)|\,|\langle w\lambda,\alpha^\vee\rangle|
\leq C_{\Phi,k}$, and hence
\[
 \bigl(1+|\alpha(X)|\,|\langle w\lambda,\alpha^\vee\rangle|\bigr)^{-1}
 \asymp_{\Phi,k}1.
\]
If $\alpha\in S\subset\Phi^+\setminus\Theta^+$, then
\eqref{eq:outside-theta-spatial-lower} and
$|\langle w\lambda,\alpha^\vee\rangle|\leq C_\Phi M$ give
\begin{equation}\label{eq:deleted-root-global-absorption}
 |\alpha(X)|^{-1}
 \leq C_{\Phi,k}M
 \bigl(1+|\alpha(X)|\,|\langle w\lambda,\alpha^\vee\rangle|\bigr)^{-1}.
\end{equation}
Moreover,
\[
 \min\left\{1,
 \bigl(|\alpha(X)|\,|\langle w\lambda,\alpha^\vee\rangle|\bigr)^{-1}
 \right\}
 \leq
 2\bigl(1+|\alpha(X)|\,|\langle w\lambda,\alpha^\vee\rangle|\bigr)^{-1}.
\]
Consequently, \eqref{eq:global-jet-intermediate} is bounded by
\[
 C_{\Phi,k}M^k
 \prod_{\alpha\in\Phi^+}
 \frac{|\langle w\lambda,\alpha^\vee\rangle|}
 {1+|\alpha(X)|\,|\langle w\lambda,\alpha^\vee\rangle|}.
\]
Summing over $w$, the ``deletion sets'' $S$, and the finitely many
Leibniz terms proves \eqref{eq:global-alternant-jets}.
\end{proof}

\begin{remark}[Dependence on the derivative directions]
Theorem~\ref{thm:jets} bounds derivatives in arbitrary unit directions using the common factor $\|\lambda\|^k$. It is natural to ask whether this estimate can be refined by retaining the dependence on the pairings $\langle w\lambda,\xi_j\rangle$ and $\alpha(\xi_j)$ in the individual Weyl summands uniformly across the spatial and spectral root hyperplanes. We leave such refinements to future work.
\end{remark}

\subsection{Euclidean spherical functions}
\label{subsec:spherical-functions}
We first recall the explicit formula for spherical functions on
complex groups and their relation to Euclidean spherical functions;
see \cite[Chapter~IV, Sections~4--5]{Helgason2000}.
For $\zeta\in\mathfrak a^*_{\mathbb C}$, define
$A_\zeta^\Phi$ by the same exponential sum as in
\eqref{eq:linear-alternant}. Put
\[
\pi^\vee(\nu)=\prod_{\alpha\in\Phi^+}
\langle\nu,\alpha^\vee\rangle,
\qquad
\pi(X)=\prod_{\alpha\in\Phi^+}\alpha(X),
\qquad N=|\Phi^+|.
\]
Let $\rho=\frac12\sum_{\alpha\in\Phi^+}\alpha$. Since every restricted root has multiplicity two, the half-sum of the positive restricted roots counted with multiplicities is $2\rho$. Define
\[
\Delta_\Phi(X)
:=\sum_{w\in W}\det(w)e^{2(w\rho)(X)}
=\prod_{\alpha\in\Phi^+}2\sinh\alpha(X).
\]
Then, initially for regular parameters,
\begin{equation}\label{eq:complex-spherical-formula}
\varphi(\zeta,X)
=
\frac{\pi^\vee(2\rho)}{\pi^\vee(i\zeta)}
\frac{A_\zeta^\Phi(X)}{\Delta_\Phi(X)},
\qquad
\varphi_0(X)=\frac{2^N\pi(X)}{\Delta_\Phi(X)}.
\end{equation}
Consequently,
\begin{equation}\label{eq:complex-motion-relation}
\varphi(\zeta,X)
=\varphi_0(X)\mathcal E_\zeta^\Phi(X),
\qquad
\mathcal E_\zeta^\Phi(X)
:=\frac{\pi^\vee(\rho)}{\pi^\vee(i\zeta)}
\frac{A_\zeta^\Phi(X)}{\pi(X)}.
\end{equation}
Here $\mathcal E_\zeta^\Phi$ is the Euclidean spherical function. 

It is classical that the apparent singularities in the formulas
for $\varphi_0$ and $\mathcal E_\zeta^\Phi$ in
\eqref{eq:complex-spherical-formula} and
\eqref{eq:complex-motion-relation} are removable, as follows
from their integral representations
\cite[Chapter~IV, Section~4]{Helgason2000}.
For completeness, we give a direct verification. We have
\begin{equation}\label{eq:varphi0}
 \varphi_0(X)=\prod_{\alpha\in\Phi^+}
 \frac{\alpha(X)}{\sinh\alpha(X)},
\end{equation}
and $t/\sinh t$ extends smoothly at $t=0$ with value $1$.
For $\mathcal E_\zeta^\Phi$, apply the BGG--Demazure identity
first in $X$ to divide $A_\zeta^\Phi(X)$ by $\pi(X)$. The resulting quotient is
entire and $W$-skew in $\zeta$. By
Remark~\ref{rem:complex-BGG}, we may apply the same identity
to the coroot system in the complex spectral variable to
divide by $\pi^\vee(\zeta)$ as well. The integral formulas
show that $\mathcal E_\zeta^\Phi(X)$ and all its spatial
derivatives are smooth in $X$, entire in $\zeta$, and jointly
continuous. These extensions are unique by density of the
regular sets.

Theorem~\ref{thm:jets} gives the following derivative estimate
for Euclidean spherical functions with real spectral parameters. 

\begin{proposition}\label{prop:global-normalized-jets}
For every integer $k\geq0$, one has
\begin{equation}\label{eq:global-normalized-jets}
 \left|
 D_{\xi_1}\cdots D_{\xi_k}
 \mathcal E_\lambda^\Phi(X)
 \right|
 \leq
 C_{\Phi,k}\|\lambda\|^k
 \mathcal R_\Phi(\lambda,X)
\end{equation}
for all $\lambda\in \mathfrak{a}^*$, $X\in \mathfrak{a}$, and unit vectors
$\xi_1,\ldots,\xi_k\in \mathfrak{a}$.
\end{proposition}

\begin{proof}
For $\lambda\in\mathfrak{a}^*$ and $w\in W$, we have
\[
 |\pi^\vee(i\lambda)|
 =|\pi^\vee(\lambda)|
 =\prod_{\alpha\in\Phi^+}
   |\langle w\lambda,\alpha^\vee\rangle|.
\]
For regular $\lambda$, we apply Theorem~\ref{thm:jets}
in \eqref{eq:complex-motion-relation} and divide each
summand by this product to obtain
\eqref{eq:global-normalized-jets}, with the fixed factor
$|\pi^\vee(\rho)|$ absorbed into $C_{\Phi,k}$.
The continuity of the spatial derivatives of
$\mathcal E_\lambda^\Phi$ in $\lambda$ allows us to approximate
any singular $\lambda$ by regular parameters and pass to the
limit, proving the estimate for every
$\lambda\in\mathfrak a^*$.
\end{proof}

\subsection{Extension to complex spectral parameters}
\label{subsec:complex-parameters}

We now extend the preceding estimate to arbitrary complex spectral parameters by the maximum principle on the
upper half-plane.

\begin{proposition} 
\label{prop:complex-normalized-jets}
For every integer $k\geq0$, there exists $C_{\Phi,k}>0$ such that
\begin{equation}\label{eq:complex-normalized-jets}
 \left|D_{\xi_1}\cdots D_{\xi_k}
 \mathcal E_\zeta^\Phi(X)\right|
 \leq C_{\Phi,k}\|\zeta\|^k
 e^{h_\Phi(\operatorname{Im}\zeta,X)}
 \mathcal R_\Phi(\zeta,X)
\end{equation}
for all $\zeta\in \mathfrak{a}^*_{\mathbb C}$, $X\in \mathfrak{a}$, and unit vectors
$\xi_1,\ldots,\xi_k\in \mathfrak{a}$.
\end{proposition}

\begin{proof}
Write $\zeta=\lambda+i\eta$ with $\lambda,\eta\in \mathfrak{a}^*$, and
put $M=\|\zeta\|$ and $h=h_\Phi(\eta,X)$.
We first assume that $X$ and $\lambda$ are regular. Fix these
parameters and the derivative directions, and define the entire
function
\[
 F(z):=D_{\xi_1}\cdots D_{\xi_k}
       \mathcal E_{\lambda+z\eta}^\Phi(X),
 \qquad z\in\mathbb C.
\]
Our goal is to bound $F(i)$.

We first record its growth in the upper half-plane. Let
$p(z):=\pi^\vee(i(\lambda+z\eta))$. Differentiating the quotient
in \eqref{eq:complex-motion-relation} gives, wherever $p(z)\neq0$,
\begin{equation}\label{eq:complex-line-expansion}
 F(z)=\frac{1}{p(z)}\sum_{w\in W}
 c_w(z)e^{i\langle w(\lambda+z\eta),X\rangle},
\end{equation}
where each $c_w$ is a polynomial of degree at most $k$, whose
coefficients depend on the fixed data
$\Phi,k,X,\lambda,\eta,\xi_1,\ldots,\xi_k$.
More explicitly, denoting
$D_J:=\prod_{j\in J}D_{\xi_j}$, with $D_\varnothing$ the identity,
we have 
\[
 c_w(z)=\pi^\vee(\rho)\det(w)
 \sum_{J\subset\{1,\ldots,k\}}D_{\{1,\ldots,k\}\setminus J}(\pi^{-1})(X)
 \prod_{j\in J}i\langle w(\lambda+z\eta),\xi_j\rangle.
\]

The polynomial $p$ is not identically zero, since
$p(0)=\pi^\vee(i\lambda)\neq0$. Then there exists a radius
$R_0=R_0(\Phi,\lambda,\eta)\geq1$ such that $1/p(z)$ is bounded outside the disk $|z|<R_0$ by a constant
depending only on $\Phi,\lambda,\eta$. For $z=x+iy$ with
$y\geq0$, each exponential in
\eqref{eq:complex-line-expansion} has absolute value at most
$e^{hy}$. Combining these bounds with the polynomial bounds for $c_w$
and the boundedness of $F$ on the compact half-disk
$\{z:|z|\leq R_0,\ \operatorname{Im}z\geq0\}$, we obtain
\begin{equation}\label{eq:complex-line-growth}
 |F(x+iy)|\leq C_0(1+|x+iy|)^k e^{hy},
 \qquad y\geq0.
\end{equation} 
Here $C_0$ may depend on
$\Phi,k,X,\lambda,\eta,\xi_1,\ldots,\xi_k$, but is independent
of $z$. 

For $w\in W$ and $\alpha\in\Phi^+$, put
\[
 a_{w,\alpha}:=|\alpha(X)|\langle w\lambda,\alpha^\vee\rangle,
 \qquad
 b_{w,\alpha}:=|\alpha(X)|\langle w\eta,\alpha^\vee\rangle.
\]
Choose $\sigma_{w,\alpha}\in\{-1,1\}$ so that
$\sigma_{w,\alpha}b_{w,\alpha}=|b_{w,\alpha}|$, taking
$\sigma_{w,\alpha}=1$ when $b_{w,\alpha}=0$, and define
\[
 \ell_{w,\alpha}(z):=
 1-i\sigma_{w,\alpha}(a_{w,\alpha}+zb_{w,\alpha}),
 \qquad
 q_w(z):=\prod_{\alpha\in\Phi^+}\ell_{w,\alpha}(z).
\]
If $b_{w,\alpha}\neq0$, then $\ell_{w,\alpha}$ has its unique zero at
$-a_{w,\alpha}/b_{w,\alpha}-i/|b_{w,\alpha}|$, strictly below
the real axis; if $b_{w,\alpha}=0$, then
$\ell_{w,\alpha}=1-ia_{w,\alpha}$ is a nonzero constant. For real $t$,
\begin{equation}\label{eq:complex-root-factor-boundary}
 |\ell_{w,\alpha}(t)|
 =\sqrt{1+(a_{w,\alpha}+tb_{w,\alpha})^2}
 \asymp 1+|\alpha(X)|\,
 |\langle w(\lambda+t\eta),\alpha^\vee\rangle|.
\end{equation}
At $z=i$, one has
\begin{equation}\label{eq:complex-root-factor-interior}
\begin{aligned}
 |\ell_{w,\alpha}(i)|
 &=\sqrt{(1+|b_{w,\alpha}|)^2+a_{w,\alpha}^2}\\
 &\asymp 1+\sqrt{a_{w,\alpha}^2+b_{w,\alpha}^2}
 =1+|\alpha(X)|\,|\langle w\zeta,\alpha^\vee\rangle|.
\end{aligned}
\end{equation}
Both comparisons hold with absolute constants.

Set
\[
 \mathcal Q(z):=\prod_{w\in W}q_w(z),
 \qquad
 \mathcal P_w(z):=\frac{\mathcal Q(z)}{q_w(z)},
 \qquad L:=|W|\,|\Phi^+|.
\]
These are polynomials of degree at most $L$, and all their zeros
lie strictly below the real axis. Proposition~\ref{prop:global-normalized-jets}
and \eqref{eq:complex-root-factor-boundary} give
\[
 |\mathcal Q(t)F(t)|
 \leq C_{\Phi,k}\|\lambda+t\eta\|^k
 \sum_{w\in W}|\mathcal P_w(t)|,
 \qquad t\in\mathbb R.
\]
For any zero $r$ of $\mathcal P_w$, one has $|i-r|\geq1$ and
\[
 |t-r|\leq|t-i|+|i-r|
 \leq2|t-i|\,|i-r|=2|t+i|\,|i-r|,
 \qquad t\in\mathbb R.
\]
Factoring $\mathcal P_w$ therefore yields
\[
 |\mathcal P_w(t)|\leq2^L|t+i|^L|\mathcal P_w(i)|.
\]
Together with $\|\lambda+t\eta\|\leq\|\lambda\|+|t|\|\eta\|\leq M|t+i|$, this gives
\begin{equation}\label{eq:complex-polynomial-boundary}
 |\mathcal Q(t)F(t)|
 \leq C_{\Phi,k}M^k|t+i|^{L+k}
 \sum_{w\in W}|\mathcal P_w(i)|.
\end{equation}

Consider
\[
 G(z):=\frac{e^{ihz}\mathcal Q(z)F(z)}{(z+i)^{L+k}}.
\]
It is holomorphic on the upper half-plane and continuous on its
closure. It is also bounded there. Indeed, applying \eqref{eq:complex-line-growth}, 
$|e^{ih(x+iy)}|=e^{-hy}$ cancels the exponential in
\eqref{eq:complex-line-growth}, while
$|\mathcal Q(z)|\leq C_1(1+|z|)^L$ and
\[
 |z+i|^2=|z|^2+2\operatorname{Im}z+1
 \geq\tfrac12(1+|z|)^2,
 \qquad \operatorname{Im}z\geq0.
\]
Thus \eqref{eq:complex-polynomial-boundary} and the maximum
principle for bounded holomorphic functions on the upper half-plane
imply
\[
 |G(z)|\leq C_{\Phi,k}M^k
 \sum_{w\in W}|\mathcal P_w(i)|,
 \qquad \operatorname{Im}z\geq0.
\]

Evaluating at $z=i$ and using $\mathcal Q(i)\neq0$ gives
\[
 |F(i)|\leq C_{\Phi,k}M^ke^h
 \sum_{w\in W}\frac{|\mathcal P_w(i)|}{|\mathcal Q(i)|}
 =C_{\Phi,k}M^ke^h\sum_{w\in W}|q_w(i)|^{-1}.
\]
By \eqref{eq:complex-root-factor-interior},
\[
\sum_{w\in W}|q_w(i)|^{-1}
\asymp_\Phi \mathcal R_\Phi(\zeta,X).
\]
This proves \eqref{eq:complex-normalized-jets}
for regular $X$ and $\operatorname{Re}\zeta$.  Approximation by regular $X$ and
$\lambda$, with $\eta$ fixed, proves the result everywhere by
continuity. 
\end{proof}

We remark that, multiplying \eqref{eq:complex-normalized-jets} by
$|\pi^\vee(i\zeta)|/|\pi^\vee(\rho)|$ and using
\eqref{eq:complex-motion-relation} gives the complex-parameter
counterpart of Theorem~\ref{thm:jets}:
\begin{equation}\label{eq:complex-alternant-jets}
\begin{aligned}
 &\left|D_{\xi_1}\cdots D_{\xi_k}
       \left(\frac{A_\zeta^\Phi}{\pi}\right)(X)\right|\\
 &\qquad\leq C_{\Phi,k}\|\zeta\|^k
 e^{h_\Phi(\operatorname{Im}\zeta,X)}
 \sum_{w\in W}\prod_{\alpha\in\Phi^+}
 \frac{|\langle w\zeta,\alpha^\vee\rangle|}
 {1+|\alpha(X)|\,|\langle w\zeta,\alpha^\vee\rangle|}.
\end{aligned}
\end{equation}

\subsection{Bounds for spherical functions and comparison with Cowling--Nevo}
\label{subsec: cowling-nevo}

\begin{proof}[Proof of Theorem~\ref{cor:spherical-derivatives}]
Each derivative of $t/\sinh t$ is bounded in absolute value by a
constant times $t/\sinh t$. Applying Leibniz's rule to
\eqref{eq:varphi0} therefore gives, for every integer $r\geq0$ and unit vectors
$\eta_1,\ldots,\eta_r$,
\begin{equation}\label{eq:phi-zero-derivatives}
 \left|D_{\eta_1}\cdots D_{\eta_r}\varphi_0(X)\right|
 \leq C_{G,r}\varphi_0(X).
\end{equation}
Apply Leibniz's rule to
$\varphi(\zeta,X)=\varphi_0(X)\mathcal E_\zeta^\Phi(X)$,
using Proposition~\ref{prop:complex-normalized-jets} for each
derivative order at most $k$ and
\eqref{eq:phi-zero-derivatives}. This gives
\begin{align*}
 \left|D_{\xi_1}\cdots D_{\xi_k}
       \varphi(\zeta,\cdot)(X)\right|
 & \leq C_{G,k}\varphi_0(X)
 e^{h_\Phi(\operatorname{Im}\zeta,X)}
 \mathcal R_\Phi(\zeta,X)\sum_{j=0}^k\|\zeta\|^j\\
 & \leq C_{G,k}(1+\|\zeta\|)^k
 \varphi_0(X)e^{h_\Phi(\operatorname{Im}\zeta,X)}
 \mathcal R_\Phi(\zeta,X),
\end{align*}
as required.
\end{proof}

We now compare the rootwise estimate
\eqref{eq:spherical-derivative-bound} with the estimates
of Cowling and Nevo. 
For $\beta\in\Phi^+$, write its simple-root expansion as
$\beta=\sum_{j=1}^r n_j(\beta)\alpha_j$. Following Cowling and Nevo
\cite[Definition 1.5]{CowlingNevo2001}, define
\[
 \gamma_\Phi:=\min_{1\leq j\leq r}
 \#\{\beta\in\Phi^+:n_j(\beta)>0\}.
\]
This exponent was also used by the second author in the study of
restriction of characters to submanifolds of compact Lie groups
\cite{ZhangCharacterRestriction}. For irreducible $\Phi$, it agrees
with the first optimal ``peeling number'' $q_{1,0}$ defined in
\cite[Lemma~3.1]{ZhangCharacterRestriction}.
For regular $H\in\mathfrak a$, put
\[
 B_\Phi(H):=\prod_{\alpha\in\Phi^+}|\alpha(H)|^{-1}.
\]
For $Y\in\mathfrak a$, define
\[
 \Omega_\Phi(Y):=\prod_{\alpha\in\Phi^+}(1+|\alpha(Y)|),
\]
and, for $\mu\in\mathfrak a^*_{\mathbb C}$, write
\[
 S_\mu^\Phi(Y):=\sum_{w\in W}e^{\langle w\mu,Y\rangle}.
\]

Theorem~\ref{cor:spherical-derivatives} recovers
\cite[Theorem~1.2]{CowlingNevo2001} for all complex spectral
parameters. In our convention, their spectral parameter is
$i\zeta=-\eta+i\lambda$ when $\zeta=\lambda+i\eta$, and
their half-sum of positive restricted roots with multiplicities is
$2\rho$. Their estimate therefore takes the form
\begin{equation}\label{eq:complex-cowling-nevo}
\begin{aligned}
 \left|\frac{\mathrm d^k}{\mathrm dt^k}
       \varphi(\zeta,tH)\right|
\leq C_{G,k}B_\Phi(H)
 \frac{(1+\|\zeta\|)^k}
      {(1+t\|\zeta\|)^{\gamma_\Phi}}
 \Omega_\Phi(tH)
 \frac{S_{-\eta}^\Phi(tH)}{S_{2\rho}^\Phi(tH)},
\end{aligned}
\end{equation}
for unit regular $H\in\mathfrak a$, $t\geq0$, and
$\zeta\in\mathfrak a^*_{\mathbb C}$.

To obtain this from the rootwise estimate, first let
$\nu\in\mathfrak a^*$ be real. For each $w\in W$, reduction to the dominant chamber and
nonnegativity of the coefficients in the simple-coroot expansions
yield at least $\gamma_\Phi$
positive roots $\alpha$ for which
$|\langle w\nu,\alpha^\vee\rangle|\geq c_\Phi\|\nu\|$. 
Since $0<|\alpha(H)|\leq C_\Phi$, \[
\begin{aligned}
 \frac{1}
 {1+|\alpha(tH)|\,|\langle w\nu,\alpha^\vee\rangle|}
 &\leq
 \frac{1}{1+c_\Phi t|\alpha(H)|\,\|\nu\|}
 &\leq
 \frac{C_\Phi}{|\alpha(H)|(1+t\|\nu\|)}.
\end{aligned}
\]
The remaining factors of the same form are at most
$1\leq C_\Phi|\alpha(H)|^{-1}$. Multiplication and summation
over $w$ give
\[
 \mathcal R_\Phi(\nu,tH)
 \leq C_\Phi B_\Phi(H)(1+t\|\nu\|)^{-\gamma_\Phi}.
\]
For $\zeta=\lambda+i\eta$, choose
$\nu\in\{\lambda,\eta\}$ with
$\|\nu\|\geq\|\zeta\|/\sqrt2$.
Since $|\langle w\zeta,\alpha^\vee\rangle|
\geq|\langle w\nu,\alpha^\vee\rangle|$, it follows that
\begin{equation}\label{eq:collapse-rootwise-factors}
 \mathcal R_\Phi(\zeta,tH)
 \leq C_\Phi B_\Phi(H)
 (1+t\|\zeta\|)^{-\gamma_\Phi}.
\end{equation}
Moreover,
\begin{align}\label{eq: S-eta}
     e^{h_\Phi(\eta,tH)}\leq S_{-\eta}^\Phi(tH)
 \leq |W|e^{h_\Phi(\eta,tH)},
\end{align}
and \cite[Lemma~2.1]{CowlingNevo2001} gives
\begin{equation}\label{eq:phi-zero-cowling-nevo}
 \varphi_0(tH)
 \leq C_G\frac{\Omega_\Phi(tH)}{S_{2\rho}^\Phi(tH)}.
\end{equation}
Combining the above three bounds with
\eqref{eq:spherical-derivative-bound}, with
$\xi_1=\cdots=\xi_k=H$, proves
\eqref{eq:complex-cowling-nevo}. 

We also recover the more precise estimate in
\cite[Corollary~2.5]{CowlingNevo2001}. Let $\Phi_0$ be a
standard parabolic root subsystem and put
$\Gamma=\Phi^+\setminus\Phi_0^+$. For $\zeta\neq0$ such that
$\langle\zeta,\beta^\vee\rangle\neq0$ for all $\beta\in\Gamma$,
we have
\begin{equation}\label{eq:collapse-subsystem-factors}
\begin{aligned}
 \mathcal R_\Phi(\zeta,tH)
 &\leq
 \frac{C_\Phi}{(1+t\|\zeta\|)^{|\Gamma|}}
 \left(
 \prod_{\beta\in\Gamma}
 \frac{\|\zeta\|}{|\langle\zeta,\beta^\vee\rangle|}
 \right)
 \sum_{w\in W}
 \prod_{\beta\in\Gamma}|(w\beta)(H)|^{-1}.
\end{aligned}
\end{equation}
Indeed, Weyl invariance of the inner product $\langle\, ,\, \rangle$ and reindexing the
positive roots give
\[
 \mathcal R_\Phi(\zeta,tH)
 =
 \sum_{w\in W}\prod_{\beta\in\Phi^+}
 \frac{1}
 {1+t|(w\beta)(H)|\,|\langle\zeta,\beta^\vee\rangle|}.
\]
We retain in the
summands only the factors indexed by $\beta\in\Gamma$.
For $\beta\in\Gamma$, the corresponding factor satisfies
\[
 \frac{1}
 {1+t|(w\beta)(H)|\,|\langle\zeta,\beta^\vee\rangle|}
 \leq
 \frac{C_\Phi}{1+t\|\zeta\|}
 \frac{\|\zeta\|}
 {|(w\beta)(H)|\,|\langle\zeta,\beta^\vee\rangle|},
\]
since $|(w\beta)(H)|\leq C_\Phi$ and
$|\langle\zeta,\beta^\vee\rangle|\leq C_\Phi\|\zeta\|$.
All remaining factors are at most one.

Up to constants depending only on $\Phi$, the spectral product
and spatial sum in \eqref{eq:collapse-subsystem-factors} are
their $B_1^*(i\zeta/\|\zeta\|)$ and $B_1(H)$, respectively.
Combining this estimate with
\eqref{eq:spherical-derivative-bound}, \eqref{eq: S-eta} and 
\eqref{eq:phi-zero-cowling-nevo} 
therefore gives \cite[Corollary~2.5]{CowlingNevo2001}.

Compared with the Cowling--Nevo estimates,
\eqref{eq:spherical-derivative-bound} retains the individual
spatial and spectral root factors. It records the transition at
$|\alpha(X)|\,|\langle w\zeta,\alpha^\vee\rangle|\asymp1$
separately for each root, with uniform control across root
hyperplanes and their intersections. In addition, it applies
to arbitrary mixed radial derivatives.

\section{Estimates for characters of compact Lie groups}

\subsection{Localization near affine root hyperplanes}

We now turn to the compact setting.  The affine root hyperplanes are
\[
 \mathcal H_{\alpha,k}
   =\{H\in\mathfrak t:\alpha(H)=2\pi k\},
 \qquad \alpha\in\Phi,\ k\in\Z.
\]
Let
\(
Q^\vee=\sum_{\alpha\in\Phi}\mathbb Z\alpha^\vee
\)
be the coroot lattice, and let
$W_{\rm aff}=W\ltimes2\pi Q^\vee$ be the affine Weyl group. We have the canonical linear-part homomorphism
$W_{\rm aff}\to W$. 

By an affine flat of the affine root arrangement we mean a
nonempty intersection of affine root hyperplanes
$\mathcal H_{\alpha,k}$, with $\mathfrak t$ regarded as the
intersection of the empty family. The following standard localization lemma allows us to replace
$|\alpha(X)|$ by $d_\alpha(H)$ in the alternant estimate; see also
\cite[Section~2.2]{ZhangCharacterRestriction} for a similar
localization argument. We include a proof for completeness.

\begin{lemma}\label{lem:affine-local}
There are constants $c_\Phi,C_\Phi>0$ such that, for every
$H\in\mathfrak t$, there is an affine flat $F$ of the affine root
arrangement with the following properties. Let $H_0\in F$ be the orthogonal projection of $H$ onto $F$, and write $X=H-H_0$. Let
\[
 \Phi_F=\{\alpha\in\Phi:
   F\subset\mathcal H_{\alpha,k}\ \text{for some }k\in\Z\}.
\]
Then
\begin{align}
 c_\Phi|\alpha(X)|
 &\le d_\alpha(H)\le C_\Phi|\alpha(X)|,
       &&\alpha\in\Phi_F,\label{eq:inside-affine}\\
 d_\alpha(H)&\ge c_\Phi,
       &&\alpha\notin\Phi_F.\label{eq:outside-affine}
\end{align}
Moreover, the pointwise stabilizer $P_F$ of $F$ in $W_{\rm aff}$ is finite,
the linear-part map $P_F\to W$ is injective, and its image $W_F$ is the Weyl
group of the root subsystem $\Phi_F$.
\end{lemma}

\begin{proof}
By an affine Weyl transformation, we assume $H$ lies in the closure $\overline{\cA}$
of a fixed fundamental alcove.  Let $\mathcal F$ be the finite collection of affine forms
\[
 a_{\alpha,k}(Y)=\alpha(Y)-2\pi k
\]
whose zero hyperplanes intersect $\overline{\cA}$.

For a subset $\mathcal S\subset\mathcal F$ whose zero hyperplanes have empty
intersection, compactness of $\overline{\cA}$ implies
\[
 \min_{Y\in\overline{\cA}}\max_{a\in\mathcal S}|a(Y)|>0.
\]
There are only finitely many such subsets.  Choose $\delta_\Phi>0$ smaller
than all these positive minima. 

Let $\mathcal S(H)$ consist of the forms $a\in\mathcal F$ with
$|a(H)|\le\delta_\Phi$.  By the choice of $\delta_\Phi$, their zero
hyperplanes have a common intersection; call it $F$. For each nonempty consistent subfamily $\mathcal S$ of
$\mathcal F$, the distance of a point $H$ to the common
zero set is bounded by a constant times  $\max_{a\in\mathcal S}|a(H)|$. Since there are only finitely
many such subfamilies, the constant can be chosen uniformly.
Consequently, the orthogonal projection $H_0$ of $H$ onto $F$
satisfies
\begin{align}\label{eq: dist}
 \|X\|=\operatorname{dist}(H,F)\leq C_\Phi\delta_\Phi.
\end{align}

If $\alpha\in\Phi_F$, then $\alpha(H_0)=2\pi k$ for some
$k\in\Z$. Thus
\[
 d_\alpha(H)=|\sin(\alpha(X)/2)|.
\]
Taking $\delta_\Phi$ sufficiently small, the bound \eqref{eq: dist} ensures that $|\alpha(X)|\leq1$ for all
$\alpha\in\Phi$. The elementary comparison
$|\sin(t/2)|\asymp|t|$ for $|t|\leq1$ therefore proves
\eqref{eq:inside-affine}.

By local finiteness of the affine root arrangement and
compactness of $\overline{\cA}$, we also choose $\delta_\Phi$
sufficiently small that
\[
 |a_{\alpha,k}(Y)|>\delta_\Phi
 \qquad
 \text{for all }Y\in\overline{\cA}
 \text{ and }a_{\alpha,k}\notin\mathcal F.
\]
Thus every affine form satisfying
$|a_{\alpha,k}(H)|\leq\delta_\Phi$ belongs to $\mathcal F$.
If $\alpha\notin\Phi_F$, then
$\operatorname{dist}(\alpha(H),2\pi\Z)>\delta_\Phi$;
otherwise there would be some $k\in\mathbb{Z}$ such that $|a_{\alpha,k}(H)|\leq\delta_\Phi$, forcing $\alpha\in\Phi_F$. This proves
\eqref{eq:outside-affine}.

For $\alpha,\beta\in\Phi_F$, crystallographic integrality implies that
$s_\alpha\beta=\beta-\langle\beta,\alpha^\vee\rangle\alpha$
is constant on $F$ with value in $2\pi\Z$. Thus $\Phi_F$ is a root
subsystem of $\Phi$. By
\cite[Ch.~V, \S~3, no.~3, Proposition~2]{BourbakiLie46},
$P_F$ is generated by the affine reflections whose hyperplanes
contain $F$. A translation sending a point of \(F\) to the origin conjugates these generating affine
reflections to the root reflections $s_\alpha$ with
$\alpha\in\Phi_F$. Thus the linear-part map identifies
$P_F$ with the Weyl group $W_F$ of $\Phi_F$; in particular,
$P_F$ is finite.
\end{proof}

\subsection{Proofs of Theorems~\ref{thm:periodic}
and~\ref{cor:character}}

\begin{proof}[Proof of Theorem~\ref{thm:periodic}]
Apply Lemma~\ref{lem:affine-local} and write $H=H_0+X$.  Decompose the finite
Weyl group into right cosets
\begin{align}\label{eq:WFu}
     W=\bigsqcup_{u\in\cD_F}W_Fu.
\end{align}
For each $p\in W_F$, let $\widetilde p\in P_F$ be its unique affine lift.
Since $\widetilde pH_0=H_0$, one has
\[
 p^{-1}H_0-H_0\in2\pi Q^\vee.
\]
If $\lambda$ is integral, so is $u\lambda$, and therefore
\begin{equation}\label{eq:phase-at-H0}
 e^{i\ip{pu\lambda}{H_0}}
 =e^{i\ip{u\lambda}{p^{-1}H_0}}
 =e^{i\ip{u\lambda}{H_0}}.
\end{equation}
The contribution of the coset $W_Fu$ is consequently
\begin{align}
 \sum_{p\in W_F}\det(pu)e^{i\ip{pu\lambda}{H}}
  =\det(u)e^{i\ip{u\lambda}{H_0}}
    A_{u\lambda}^{\Phi_F}(X).
\label{eq:affine-coset-block}
\end{align}
We apply Theorem~\ref{thm:linear} with root system $\Phi_F$
and spectral parameter $u\lambda$ to establish, using \eqref{eq:inside-affine}, 
\begin{align}
 \left|A_{u\lambda}^{\Phi_F}(X)\right|
 \le C_\Phi\sum_{p\in W_F}
 \prod_{\alpha\in\Phi_F^+}
 \min\{1,d_\alpha(H)
          |\langle pu\lambda,\alpha^\vee\rangle|\}.
\label{eq:subsystem-bound}
\end{align}

It remains to insert the roots outside $\Phi_F$.  For
$\alpha\notin\Phi_F$, \eqref{eq:outside-affine} gives
$d_\alpha(H)\ge c_\Phi$.  Regularity and integrality give
\( |\langle pu\lambda,\alpha^\vee\rangle|
 \in\Z\setminus\{0\}\), 
so
\begin{equation}\label{eq:missing-factor}
 \min\{1,d_\alpha(H)
        |\langle pu\lambda,\alpha^\vee\rangle|\}
 \ge \min\{1,c_\Phi\}>0.
\end{equation}
Thus, after enlarging $C_\Phi$ if necessary, we may insert the
factors corresponding to roots outside
$\Phi_F$ into the product in \eqref{eq:subsystem-bound}. Summing the resulting bounds over $u\in\cD_F$ proves \eqref{eq:periodic-main}.
\end{proof}

\begin{proof}[Proof of Theorem~\ref{cor:character}]
Put
\[
 \Lambda=\mu+\rho,
 \qquad
 M=\|\Lambda\|.
\]
The weight $\Lambda$ is regular integral, and hence $M\geq c_\Phi>0$.
Fix $H\in\mathfrak t$, apply Lemma~\ref{lem:affine-local}, and write
$H=H_0+X$.  Let $\operatorname{pr}_F$ denote orthogonal projection onto
the affine subspace $F$, and put
\[
 X_F(Y)=Y-\operatorname{pr}_F(Y).
\]
Thus $X_F(H)=X$ and $X_F(Y)\in\operatorname{span}_{\mathbb R}(\Phi_F)$.  The argument giving
\eqref{eq:affine-coset-block}, now applied with
$\operatorname{pr}_F(Y)\in F$, gives the local identity
\begin{equation}\label{eq:affine-coset-block-variable}
 A_\Lambda^\Phi(Y)
 =\sum_{u\in\cD_F}\det(u)
 e^{i\ip{u\Lambda}{\operatorname{pr}_F(Y)}}
 A_{u\Lambda}^{\Phi_F}(X_F(Y)).
\end{equation}

Write
\[
 \pi_F(Z)=\prod_{\alpha\in\Phi_F^+}\alpha(Z),
 \qquad
 g_F(Y)=\frac{A_\rho^\Phi(Y)}{\pi_F(X_F(Y))}.
\]  
Up to
a nonzero constant, the Weyl denominator formula expresses $g_F$ locally
as
\[
 \prod_{\alpha\in\Phi_F^+}
 \frac{\sin(\alpha(X_F(Y))/2)}{\alpha(X_F(Y))}
 \prod_{\alpha\in\Phi^+\setminus\Phi_F^+}
 \sin(\alpha(Y)/2).
\]
By \eqref{eq:inside-affine} and \eqref{eq:outside-affine},
the factors in this product are bounded away from zero
at $H$, where $\sin(t/2)/t$ is interpreted
as $1/2$ at $t=0$. These factors have uniformly bounded
derivatives of every fixed order. It follows that $g_F$
is nonvanishing near $H$ and, for $0\leq j\leq k$
and unit directions $\eta_1,\ldots,\eta_j\in\mathfrak{t}$,
\begin{equation}\label{eq:affine-denominator-jets}
 \left|
 D_{\eta_1}\cdots D_{\eta_j}g_F^{-1}(H)
 \right|
 \leq C_{\Phi,k}.
\end{equation}

For $u\in\cD_F$, apply Theorem~\ref{thm:jets} to the root subsystem
$\Phi_F$ with spectral parameter $u\Lambda$, noting that
$\|u\Lambda\|=M$. For $0\leq j\leq k$ and unit directions
$\zeta_1,\ldots,\zeta_j\in\mathfrak{t}$, this gives
\begin{align}
 \left|
 D_{\zeta_1}\cdots D_{\zeta_j}
 \left(\frac{A_{u\Lambda}^{\Phi_F}}{\pi_F}\right)(X)
 \right|
\leq C_{\Phi,k}M^j
 \sum_{p\in W_F}
 \prod_{\alpha\in\Phi_F^+}
 \frac{
  |\langle pu\Lambda,\alpha^\vee\rangle|
 }{
  1+|\alpha(X)|\,|\langle pu\Lambda,\alpha^\vee\rangle|
 }.
\label{eq:affine-normalized-subsystem-jets}
\end{align}
The constant may be taken to depend only on $\Phi$ and $k$, since
$\Phi_F$ ranges over finitely many root subsystems of $\Phi$.

Dividing \eqref{eq:affine-coset-block-variable} by
$\pi_F(X_F(Y))$ and multiplying by $g_F(Y)^{-1}$ gives
\[
 (\chi_\mu\circ\exp)(Y)
 =
 g_F(Y)^{-1}
 \sum_{u\in\cD_F}\det(u)
 e^{i\ip{u\Lambda}{\operatorname{pr}_F(Y)}}
 \left(\frac{A_{u\Lambda}^{\Phi_F}}{\pi_F}\right)(X_F(Y))
\]
for $Y$ near $H$. We now apply the Leibniz rule to this identity. Each derivative of the exponential phase $e^{i\ip{u\Lambda}{\operatorname{pr}_F(Y)}}$ costs at
most $C_\Phi M$. It follows from \eqref{eq:affine-denominator-jets} and \eqref{eq:affine-normalized-subsystem-jets} that
\begin{align}
 &\left|
 D_{\xi_1}\cdots D_{\xi_k}
 (\chi_\mu\circ\exp)(H)
 \right|
 \notag\\
 &\qquad\leq C_{\Phi,k}M^k
 \sum_{u\in\cD_F}\sum_{p\in W_F}
 \prod_{\alpha\in\Phi_F^+}
 \frac{
  |\langle pu\Lambda,\alpha^\vee\rangle|
 }{
  1+|\alpha(X)|
  |\langle pu\Lambda,\alpha^\vee\rangle|
 }.
\label{eq:character-local-jet-bound}
\end{align}

By \eqref{eq:inside-affine}, we have
\[
 \frac{
  \abs{\ip{pu\Lambda}{\alpha^\vee}}
 }{
  1+|\alpha(X)|\abs{\ip{pu\Lambda}{\alpha^\vee}}
 }
 \leq C_\Phi
 \frac{
  \abs{\ip{pu\Lambda}{\alpha^\vee}}
 }{
  1+d_\alpha(H)\abs{\ip{pu\Lambda}{\alpha^\vee}}
 },
 \qquad \alpha\in\Phi_F^+.
\]
If $\alpha\notin\Phi_F$, regularity and integrality give
$\abs{\ip{pu\Lambda}{\alpha^\vee}}\geq1$, while
$d_\alpha(H)\leq1$; hence
\[
 \frac{
  \abs{\ip{pu\Lambda}{\alpha^\vee}}
 }{
  1+d_\alpha(H)\abs{\ip{pu\Lambda}{\alpha^\vee}}
 }
 \geq\frac12.
\]
After enlarging $C_{\Phi,k}$, we may therefore insert the factors corresponding to roots outside $\Phi_F$ into each product in \eqref{eq:character-local-jet-bound}, and \eqref{eq:character-main} follows from the decomposition \eqref{eq:WFu}.
\end{proof}
\begin{remark}[Comparison with Hare's character bounds]
\label{rem:character-comparison}
For $H\in\mathfrak t$, put
\[
\Phi_H^+
:=\{\alpha\in\Phi^+:\alpha(H)\in2\pi\mathbb Z\}.
\]
The case $k=0$ of Theorem~\ref{cor:character} gives
\begin{equation}\label{eq:character-normalized-rootwise}
|\chi_\mu(\exp H)|
\leq C(H,\mu)
\sum_{w\in W}
\prod_{\alpha\in\Phi_H^+}
|\langle w(\mu+\rho),\alpha^\vee\rangle|,
\end{equation}
where
\[
C(H,\mu)
:=C_U\max_{w\in W}
\prod_{\alpha\in\Phi^+\setminus\Phi_H^+}
\frac{\abs{\ip{w(\mu+\rho)}{\alpha^\vee}}}{1+ d_\alpha(H)\abs{\ip{w(\mu+\rho)}{\alpha^\vee}}}.
\]
Hare obtained bounds of the same form as \eqref{eq:character-normalized-rootwise} \cite[the last inequality on p.~5]{Hare1998} (see also \cite[Eq.~(2.3)]{HareWilsonYee2000}) with $C(H,\mu)$ replaced by the larger factor
$$
C_U
\prod_{\alpha\in\Phi^+\setminus\Phi_H^+}
d_\alpha(H)^{-1},
$$
which may become unbounded as $H$ approaches additional affine root
hyperplanes.
\end{remark}

\begin{remark}[$L^p$ bounds for characters]
\label{rem:character-Lp}
For $k=0$, \eqref{eq:character-main} and the Weyl integration formula
reduce the derivation of $L^p$ upper bounds to estimating integrals
of explicit root factors over regions determined by the affine root
hyperplanes of the maximal torus. This extends the $\mathrm{SU}(3)$ strategy of
\cite[Section~5]{ZhangSU3}, retaining dependence on the relative sizes
of the spectral coordinates uniformly over all dominant highest weights.
On the other hand, sharpness requires matching lower bounds. Once sufficiently large
character values are established at suitable points, derivative estimates can show that they persist on neighborhoods of
controlled size, yielding $L^p$ lower bounds by
integration. We leave this analysis to future work.
\end{remark}

\bibliographystyle{amsplain}
\bibliography{weyl}

\end{document}